\documentclass[10pt,reqno]{amsart}
\calclayout

\usepackage{amsmath,amssymb,amsfonts,amsthm,comment,,xcolor}
\usepackage{mathrsfs}
\usepackage{faktor}
\usepackage{graphics}
\usepackage{graphicx}
\graphicspath{ {images/} }
\usepackage{bigints} 

\usepackage[colorlinks,citecolor=blue, linkcolor=blue]{hyperref}
\usepackage{float}
\usepackage{tikz}
\usepackage{tikz-cd}
\usetikzlibrary{arrows, matrix}
\usetikzlibrary{calc,angles,positioning,intersections,quotes,decorations.markings}
\usepackage{relsize}
\usepackage{enumitem}
\usepackage{mathtools}

\def\R{\mathbb{R}}
\def\C{\mathbb{C}}
\def\D{\mathbb{D}}
\def\N{\mathbb{N}}
\newcommand{\abs}[1]{\left\vert #1 \right\vert} 
\newcommand{\norm}[1]{\left\Vert #1\right\Vert}
\renewcommand\d[1]{\:\mathrm{d}#1}

\newtheorem{theorem}{Theorem}
\newtheorem{lemma}{Lemma}[section]
\newtheorem{proposition}[lemma]{Proposition}

\theoremstyle{definition}

\theoremstyle{remark}

\numberwithin{equation}{section}

\allowdisplaybreaks

\begin{document}


\title{Fourier Representations in Bergman Spaces}

\begin{abstract} We consider a class of domains, generalizing the upper half-plane, and  admitting rotational,  translational and scaling symmetries, analogous to the half-plane. We prove Paley-Wiener type representations of functions in 
Bergman spaces of such domains with respect to each of these three groups of symmetries.  The Fourier series, Fourier integral and Mellin integral representations so obtained may be used to give representations of the Bergman kernels of these domains.

\end{abstract}


\author{Debraj Chakrabarti}
\address{Department of Mathematics, Central Michigan University, Mount Pleasant, MI 48859, U.S.A.}
\email{chakr2d@cmich.edu}

\author{Pranav Upadrashta}
\address{Department of Mathematics, Central Michigan University, Mount Pleasant, MI 48859, U.S.A.}
\email{upadr1pk@cmich.edu}




\maketitle

\section{Introduction}
\subsection{Polynomial half-spaces and ellipsoids}\label{sec-preliminaries}

A well-known classical theorem of Paley and Wiener (\cite{paleywiener}; also see  \cite[Theorem 19.2]{rudin}) states that 
every function in the Hardy space $H^2(U)$ of the upper half plane $U\subset \C$ arises as the holomorphic Fourier transform of a square integrable function on the positive real line $(0,\infty)$. Similar results are known for Hardy spaces of other domains with continuous automomorphisms (\cite{bochner1944,koranyistein, Ogden}).  The goal of this paper is to study some analogs of this result for Bergman spaces of certain domains in $\C^n$ generalizing the upper half plane and  admitting large groups of automorphisms.

To define the domains which we will consider,
let $m=(m_1,\cdots,m_n)$ be a tuple of positive integers. Given a tuple $\alpha \in \N^n$ of nonnegative integers define the \emph{weight of $\alpha$ with respect to $m$} to be
\begin{equation*}
\textrm{wt}_{m}(\alpha) := \sum_{i=1}^{n} \frac{\alpha_i}{2m_i}.
\end{equation*}
A real polynomial $p: \C^{n}\to \R$ is then called a \textit{weighted homogeneous balanced polynomial} (with respect to the tuple $m\in \mathbb{N}^{n}$) if $p$ is of the form
\begin{equation}\label{eq-p}
p(w_1,\cdots, w_{n}) = \sum_{\textrm{wt}_{m}(\alpha)=\textrm{wt}_{m}(\beta)=1/2} C_{\alpha,\beta}w^{\alpha}\overline{w}^{\beta}.
\end{equation}

Let $p:\C^n\to \R$ be a weighted homogeneous balanced polynomial such that $p\geq 0$ on $\C^n$. The \textit{polynomial half space} $\mathcal{U}_p$ defined by $p$ is the unbounded domain in $\C^{n+1}$ 
\begin{equation}\label{Updef}
\mathcal{U}_p = \{ (z,w)\in \C\times \C^{n}\; |\; \textrm{Im}\; z> p(w)\}.
\end{equation}
The polynomial half space $\mathcal{U}_p$ is biholomorphically equivalent to the bounded domain $\mathcal{E}_p\subset \C^{n+1}$, where
\begin{equation}\label{eq-Epdef}
\mathcal{E}_p = \{(z,w)\in \C\times \C^n \;|\;\abs{z}^2 + p(w) < 1\}.
\end{equation}
We call the domain $\mathcal{E}_p$ a \textit{polynomial ellipsoid}. The map $\Lambda: \mathcal{U}_p\to \mathcal{E}_p$ given by 
\begin{align}\label{eq-lambda-auto-def}
\Lambda(z,w_1,\dots,w_n) &= \left( \frac{1+iz/4}{1-iz/4}, \frac{w_1}{(1-iz/4)^{1/m_1}}, \cdots,\frac{w_n}{(1-iz/4)^{1/m_n}}\right)
\end{align}
is a biholomorphic equivalence between $\mathcal{U}_p$ and $\mathcal{E}_p$.
The unit ball in $\C^{n+1}$ is a familiar example of polynomial ellipsoid corresponding to the weighted homogeneous balanced polynomial $p(w)=\abs{w}^2$.
A special case of this is the upper half plane $U$ in $\C$ (polynomial half-space corresponding to $p\equiv 0$ on $\C^0$) and its bounded model, the unit disc (the corresponding polynomial ellipsoid).  For the ball, the map $\Lambda$ reduces to the familiar Cayley map which maps the unit ball biholomorphically onto Siegel upper half space, which is a polynomial half space.

Henceforth, $m=(m_1,\dots,m_n)$ will denote a tuple of positive integers and $p$ will be a nonnegative weighted homogeneous balanced polynomial with respect to $m$, as in \eqref{eq-p}. The significance of these domains in complex analysis is
underlined by a  classical result of Bedford and Pinchuk (see \cite{bedford1994}) : {\em Let $\Omega$ be a bounded convex domain with smooth boundary, and of finite type in sense of D'Angelo (cf.  \cite[p. 118]{d1993several}). Then $\mathrm{Aut}(\Omega)$ is noncompact if and only if $\Omega$ is biholomorphic to a polynomial ellipsoid. } Therefore, for these polynomial half spaces (and therefore polynomial ellipsoids), we have available one-parameter groups of biholomorphic automorphisms, with respect to which we can try to construct Fourier representations analogous to
the theorem of Paley and Wiener stated above.  In fact, the automorphism groups of these domains admit at least {\em three } one-parameter subgroups:

\noindent{(1) \bf Rotations:} For $\theta \in \R$, the map $\sigma_{\theta}:\mathcal{E}_p\to \mathcal{E}_p$ given by
\begin{equation*}
\sigma_{\theta}(z,w) = (e^{i\theta}z,w), \; \text{for }  (z,w) \in \mathcal{E}_p \subset \C\times \C^n
\end{equation*}
is clearly an automorphism, called  a \emph{rotation} of $\mathcal{E}_p$, which form a compact one-parameter subgroup (isomorphic to the circle group) of the group $\mathrm{Aut}(\mathcal{E}_p)$. The biholomorphic equivalence $\Lambda$ of \eqref{eq-lambda-auto-def} therefore induces a corresponding automorphism of $\mathcal{U}_p$. When $\mathcal{E}_p$ is the unit disc $\D$, these automorphisms
are simply the rotations $z\mapsto e^{i\theta} z$.

\noindent {(2) \bf Translations:}  For $\theta \in \R$, the map $\tau_{\theta}: \mathcal{U}_p \to \mathcal{U}_p$ given by
\begin{equation}\label{eq-tautheta}
\tau_{\theta}(z,w) = (z+\theta, w), \; \text{for } (z,w) \in \mathcal{U}_p \subset \C \times \C^n
\end{equation}
is an automorphism of $\mathcal{U}_p=\{(z,w)\in \C\times \C^n\;|\; \mathrm{Im}\,z>p(w)\}$, since $\mathrm{Im}\,(z+\theta)=\mathrm{Im}\,z>p(w)$.  We call $\tau_{\theta}$ a \emph{translation} of $\mathcal{U}_p$. The translations form a one-parameter subgroup of $\mathrm{Aut}(\mathcal{U}_p)$ isomorphic to $\R$. In the upper half plane $U\subset \C$, the translations are simply the maps
$z\mapsto z+\theta$ for $\theta\in \R$

\noindent { (3) \bf Scalings:} Recall that $m=(m_1,\dots,m_n)$ is the tuple of positive integers with respect to which $p$ is a weighted homogeneous balanced polynomial. For $\theta >0$, the map $\rho_{\theta}: \mathcal{U}_p\to \mathcal{U}_p$ given by 
\begin{equation*}\label{eq-rhotheta}	
\rho_{\theta}(z,w) = \left(\theta z, \theta^{1/2m_1}w_1,\cdots, \theta^{1/2m_n}w_n\right) \; \text{for } (z,w) \in \mathcal{U}_p 		
\end{equation*}
will be called a a \emph{scaling} of $\mathcal{U}_p$. Notice that for the weighted
homogeneous balanced polynomial $p$ given by \eqref{eq-p} we have
\[ p\left(\theta^{1/2m_1}w_1,\cdots, \theta^{1/2m_n}w_n\right)= \theta p(w),\]
so $\rho_\theta$ is 
an automorphism of $\mathcal{U}_p$.  The scalings of $\mathcal{U}_p$ form a one-parameter subgroup of $\mathrm{Aut}(\mathcal{U}_p)$ isomorphic to the multiplicative group of positive real numbers.  In case of the upper half plane $U\subset \C$,the scalings are precisely the dilations $z\mapsto \theta z$, where $\theta>0$.

\subsection{Fourier representations and applications}\label{Results}

For a domain $\Omega \subset \C^n$,  and for a continuous function $\lambda>0$ on $\Omega$, let $A^2(\Omega, \lambda)$ denote the weighted Bergman space corresponding to $\lambda$, i.e.,
\begin{equation*}\label{eq-weightedbergman}
    A^2(\Omega, \lambda)= \left\{f\in \mathcal{O}(\Omega)| \int_\Omega \abs{f}^2\lambda dV <\infty\right\}.
\end{equation*}
When $\lambda\equiv1$, we denote the corresponding space by $A^2(\Omega)$, the {\em Bergman space} of
square integrable holomorphic functions.
If $\phi:\Omega_1\to \Omega_2$
is a biholomorphism, recall that  we have an induced isometric isomorphism $\phi^*:A^2(\Omega_2)\to A^2(\Omega_1)$  given by 
\begin{equation}\label{eq-unitaryrep}
\phi^*f = (f\circ \phi )\cdot \det \phi', \quad \text{for all } f\in A^2(\Omega_2).
\end{equation}
In particular, a biholomorphic automorphism  of $\Omega$ induces
an isometric isomorphism of $A^2(\Omega)$ with itself . Therefore the map $\phi \mapsto \phi^*$ is a unitary representation of the group $\mathrm{Aut}(\Omega)$ of biholomorphic automorphisms of $\Omega$ in the Hilbert space $A^2(\Omega)$. When $\Omega$ is 
a polynomial half-space $\mathcal{U}_p$, this allows us to obtain Paley-Wiener type representations of $A^2(\mathcal{U}_p)$ corresponding to each of the three 
one-parameter subgroups described above.
These explicit Fourier representations
are presented in  Theorem~\ref{thm-PWCompact} (for the rotation group), Theorem~\ref{PWTranslation} (for the translation group)  and Theorem~\ref{PWDilation} (for the scaling group).

For example, in  Theorem~\ref{PWTranslation}, we represent  
each function $F\in A^2(\mathcal{U}_p)$ as a Fourier integral (see \eqref{eq-pwtranslationellipsoid} below)
\[ F(z,w)= \int_0^\infty f(t,w) e^{i2\pi zt}dt\]
where $z\in \C$ and $w\in \C^n$ are such that $(z,w)\in \mathcal{U}_p$, and the function $f$ belongs to a customized Hilbert space $\mathcal{H}_p$ of measurable functions on $(0,\infty)\times \C^n$ which are square integrable with respect to a weight depending on the geometry of the domain $\mathcal{U}_p$ and which are holomorphic in the second variable (see \eqref{eq-hpnorm} and \eqref{eq-hphol} below). 
Moreover (and this is the point of the exercise) the map $T_S$ which takes the function 
$f\in \mathcal{H}_p$ to $F\in A^2(\mathcal{U}_p)$ is not only an isometric isomorphism of Hilbert spaces, but also respects the natural action of the additive group $\R$ on the two Hilbert spaces $\mathcal{H}_p$ and $A^2(\mathcal{U}_p)$. More precisely
for each $\theta \in \R$, the following diagram commutes,
\[
\begin{tikzcd}
\mathcal{H}_p \arrow[r,"T_S"] \arrow[d,"\chi_{\theta}"] 
& A^2(\mathcal{U}_p) \arrow[d,"\tau_\theta^*"] \\
\mathcal{H}_p \arrow[r,"T_S"]
& A^2(\mathcal{U}_p)
\end{tikzcd}
\]
where $\tau_\theta^*$ is the unitary transformation of the Hilbert space $A^2(\mathcal{U}_p)$ induced by the translation automorphism  $\tau_\theta$ of 
$\mathcal{U}_p$ given by \eqref{eq-tautheta}, and $\chi_\theta:\mathcal{H}_p\to \mathcal{H}_p$ is the unitary map given by
\[ (\chi_\theta f)(t,w)= e^{i2\pi \theta t}f(t,w).\]
Therefore, $T_S$ simultaneously diagonalizes the commuting one-parameter family of unitary operators $\tau_\theta^*, \theta\in \R$ (i.e. a unitary representation of $\R$ in $A^2(\mathcal{U}_p)$). The existence of such a simultaneous diagonalization follows from abstract results of functional analysis (Stone's theorem, see \cite{hall} or \cite[Theorem VIII.8]{reedsimon}), but our goal here is to give an explicit construction of the diagonalization. Theorem~\ref{PWTranslation} extends and generalizes similar results found scattered in the literature (cf.\cite{rothaus1960domains,koranyi1962bergman, genchev,saitoh, duren2007, peloso}). In \cite{vasilevski}, the Bergman space of the Siegel upper half-space was isometrically represented 
by an $L^2$ space, respecting the action of a maximal abelian subgroup; however, this representation, unlike ours,  is not holomorphic in the complex parameter.

In Theorem~\ref{PWDilation}, we prove a similar representation theorem for the Bergman space $A^2(\mathcal{U}_p)$, but this time the scaling group replaces the translation group. Again, we have an isometric isomorphism of $A^2(\mathcal{U}_p)$ with a customized Hilbert space $\mathcal{X}_p$ which respects the action of the multiplicative group of positive reals on the two spaces. The representation is now in terms of a { Mellin } integral (see \eqref{eq-pwdilation} below). We believe that these representations for polynomial half-spaces 
have not been noted in the literature before (see however \cite{peloso}.)

The rotation group being compact, in this case (Theorem \ref{thm-PWCompact}) we obtain a Fourier representation of $A^2(\mathcal{E}_p)$ with respect to this group as an infinite series instead of as an integral for the other two cases. Note also that in this case it is more convenient to 
use the bounded biholomorphically equivalent model $\mathcal{E}_p$. Again, we have a Hilbert space $\mathcal{Y}_p$ and an isometric isomorphism with $A^2(\mathcal{E}_p)$ which respects the action of the rotation group.

In Section \ref{sec-rkhs} we obtain integral and series representations of the Bergman kernels of polynomial half spaces in terms of the reproducing kernels of certain ``direct integrands" of the spaces $\mathcal{H}_p, \mathcal{X}_p$ and $\mathcal{Y}_p$. The integral representation corresponding to $\mathcal{H}_p$ (\eqref{eq-Haslinger1} below) was already obtained by F. Haslinger (\cite{haslinger}), by differentiating a similar formula for the Szeg\"o kernel. Here we use Fourier techniques directly on Bergman spaces to recapture this formula and obtain new ones corresponding to the other two one-parameter groups.

\subsection{Acknowledgments} We would like to thank Sivaram Narayan and David Barrett for their helpful comments.  Research of the first  author was supported by a National Science Foundation grant (\#1600371), and  by a collaboration grant from the Simons Foundation (\# 316632).
\section{Fourier representation associated to the rotation group}\label{sec-rot}
Recall that throughout this paper $p$ is a fixed nonnegative weighted homogeneous balanced polynomial (cf. \eqref{eq-p}).
For $n\geq 1$ let $\mathbb{B}_p \subset \C^n$,  be the domain given by
\begin{equation}\label{eq-Bpdef}
\mathbb{B}_p = \{ w\in \C^n \; |\; p(w) < 1\}.
\end{equation}
For $k\in \N$, let $\mathcal{W}_p(k)$ be the weighted Bergman space on $\mathbb{B}_p$ with respect to the weight $w\mapsto (1-p(w))^{k+1}$, i.e., 
\begin{equation}\label{eq-wpk}
\mathcal{W}_p(k)  = A^2\left(\mathbb{B}_p,(1-p)^{k+1}\right)=\left\{f\in \mathcal{O}(\mathbb{B}_p)| \int_{\mathbb{B}_p} \abs{f}^2 (1-p)^{k+1}dV <\infty\right\}.
\end{equation}
When $n=0$, we set  $\mathcal{W}_p(k)=\C$ for all $k\in \N$.
Let $\mathcal{Y}_p$ be the Hilbert space of sequences $a=(a_k)_{k=0}^{\infty}$ where for each $k\in \N$, $a_k\in \mathcal{W}_p(k)$ and
\begin{equation}\label{eq-normYp}
 \norm{a}_{\mathcal{Y}_p}^2:=\pi \sum_{k=0}^{\infty}  \frac{1}{k+1}\norm{a_k}_{\mathcal{W}_p(k)}^2=\pi \sum_{k=0}^{\infty}\frac{1}{k+1}\int_{\mathbb{B}_p} \abs{a_k(w)}^2(1-p(w))^{k+1}dV(w)< \infty.
\end{equation}
The following result describes the Fourier representation of
$A^2(\mathcal{U}_p)$ with respect to the compact group of rotations. As expected, we have a series representation, rather than an integral representation, and it is more convenient to use the bounded model $\mathcal{E}_p$:

\begin{theorem}\label{thm-PWCompact}
The map $T: \mathcal{Y}_p \to A^2(\mathcal{E}_p)$ given by 
\begin{equation}\label{eq-PWCompact}
Ta(z,w) = \sum_{k=0}^{\infty} a_k(w)z^k, \quad \text{for all  } (z,w) \in \mathcal{E}_p
\end{equation}
is an isometric isomorphism of Hilbert spaces.
\end{theorem}

\begin{proof}
We begin by noting the following:
let $\D(r) =\{z\in \C \;|\; \abs{z}<r\}$, and let $\ell^2_r$ be the Hilbert space of complex sequences $a=(a_k)_{k=0}^{\infty}$ such that
\[
\norm{a}_{\ell^2_r}^2:=\pi \sum_{k=0}^{\infty} \frac{r^{2k+2}}{k+1}\abs{a_k}^2 < \infty.
\]
Then it follows by Parseval's formula and the easily established fact that 
the monomials $\{z^j\}$ form a complete orthogonal sequence in the Hilbert space $A^2(\D(r))$ that the map from $\ell^2_r$  to $A^2(\D(r))$ given by
\begin{equation}\label{eq-PWdisk}
a=(a_k)_{k=0}^{\infty} \longmapsto \sum_{k=0}^{\infty}a_kz^k
\end{equation}
is an isometric isomorphism of Hilbert spaces.

Now for $a\in \mathcal{Y}_p$, we may interchange summation and integration in \eqref{eq-normYp} by monotone convergence theorem. Then by Fubini's theorem, for almost all $w\in \mathbb{B}_p$  
\begin{equation}\label{eq-tonellicompact}
\pi \sum_{k=0}^{\infty} \frac{(1-p(w))^{k+1}}{k+1}\abs{a_k(w)}^2 < \infty.
\end{equation}
\noindent If $w \in \mathbb{B}_p$ is such that \eqref{eq-tonellicompact} holds, it follows from equation \eqref{eq-PWdisk} that the function $Ta(\cdot,w)$ given by the right hand side of \eqref{eq-PWCompact} is in $A^2(\D(\sqrt{1-p(w)}))$ and 
\begin{equation}\label{eq-compactplancherel}
\pi \sum_{k=0}^{\infty} \frac{(1-p(w))^{k+1}}{k+1}\abs{a_k(w)}^2 = \int_{\D(\sqrt{1-p(w)})} \abs{Ta(z,w)}^2 \d V(z). 
\end{equation}
Integrating \eqref{eq-compactplancherel} over $\mathbb{B}_p$ we get 
\begin{align}
\norm{Ta}_{L^2(\mathcal{E}_p)}^2 &=\int_{\mathbb{B}_p} \int_{\D(\sqrt{1-p(w)})} \abs{Ta(z,w)}^2 \d V(z) \d V(w) \nonumber \\
                                 &= \pi \sum_{k=0}^{\infty} \frac{1}{k+1}\int_{\mathbb{B}_p} \abs{a_k(w)}^2(1-p(w))^{k+1} \d V(w) =\norm{a}_{\mathcal{Y}_p}^2.  \label{eq-compactisometry} 
\end{align}
Since the partial sums of \eqref{eq-PWCompact} are holomorphic, it follows that 
$T$ is an isometry from $\mathcal{Y}_p$ into $A^2(\mathcal{E}_p)$ and consequently injective.
To show that  $T$ is  surjective,  let $F\in A^2(\mathcal{E}_p)$. Then we have 
\[
\norm{F}_{A^2(\mathcal{E}_p)}^2=\int_{\mathbb{B}_p} \int_{\D(\sqrt{1-p(w)})} \abs{F(z,w)}^2 \d V(z) \d V(w) < \infty.
\]
Thus, for almost all $w\in \mathbb{B}_p$, the inner integral in the above equation over $\D(\sqrt{1-p(w)})$ is finite.
For such a $w\in \mathbb{B}_p$, by 
equation \eqref{eq-PWdisk}, there is an $a(w) \in \ell^2_{\sqrt{1-p(w)}}$ such that 
\[
F(z,w) = \sum_{k=0}^{\infty} a_k(w)z^k, \quad \text{for all }z\in \D(\sqrt{1-p(w)}) 
\]
and 
\begin{equation}\label{eq-compactplancherelsurj}
\int_{\D(\sqrt{1-p(w)})}\abs{F(z,w)}^2 \d V(z) = \pi \sum_{k=0}^{\infty} \frac{(1-p(w))^{k+1}}{k+1}\abs{a_k(w)}^2.
\end{equation}
Using the Cauchy integral formula applied to $F(\cdot,w)$ in the disk $\D(\sqrt{1-p(w)})$, it follows that $a_k \in \mathcal{O}(\mathbb{B}_p)$ for each $k\in \N$. Integrating  both sides of equation \eqref{eq-compactplancherelsurj} on $\mathbb{B}_p$ yields
\begin{align*}
\norm{a}_{\mathcal{Y}_p}^2 &=   \int_{\mathbb{B}_p} \pi \sum_{k=0}^{\infty}   \frac{(1-p(w))^{k+1}}{k+1} \abs{a_k(w)}^2\d V(w) \\
                            &= \int_{\mathbb{B}_p} \int_{\D(\sqrt{1-p(w)})} \abs{F(z,w)}^2 \d V(z) \d V(w) <\infty.
\end{align*}
This shows that $a\in \mathcal{Y}_p$ and that $T$ is surjective. 
\end{proof}

\section{Fourier representation with respect to the translation group}\label{sec-trans}

Let $\mathcal{H}_p$ be the Hilbert space of measurable functions $g$ on $(0,\infty)\times \C^n$ such that
\begin{equation}\label{eq-hpnorm}
\norm{g}_{\mathcal{H}_p}^2:=\int_{\C^n} \int_0^{\infty}\abs{g(t,w)}^2 \frac{e^{-4\pi p(w)t}}{4\pi t} \d V(w) \d t < \infty,
\end{equation}
and
\begin{equation}\label{eq-hphol}
    \frac{\partial g}{\partial \overline{w}_j}=0  \text{ in the sense of distributions, } 1\leq j\leq n,
\end{equation}
where $w_1,\dots,w_n$ are co-ordinates of $\C^n$. In other words, functions in $\mathcal{H}_p$ are square integrable on $(0,\infty)\times \C^n$ with respect to the weight $(t,w)\mapsto e^{-4\pi p(w)t}/4\pi t$, and are holomorphic in the variable $w\in \C^n$. We can now state a representation theorem for functions in 
$A^2(\mathcal{U}_p)$ which respects the translation group:

\begin{theorem}\label{PWTranslation}
The map $T_S: \mathcal{H}_p \to A^2(\mathcal{U}_p)$ given by 
\begin{equation}\label{eq-pwtranslationellipsoid}
T_Sf(z,w) = \int_0^{\infty} f(t,w) e^{i2\pi zt} \d t, \quad \text{for all } (z,w) \in \mathcal{U}_p
\end{equation}
is an isometric isomorphism of Hilbert spaces.
\end{theorem}
\subsection{The case of the strip} As a first step in the proof of Theorem~\ref{PWTranslation}, we construct a Fourier representation of functions in the Bergman space of a horizontal strip in $\C$ which has translational symmetries though it is not a polynomial half space.  For $-\infty \leq a <b \leq \infty$, let $S(a,b)$ be the strip
\begin{equation}\label{eq-stripdef}
S(a,b) = \{ z\in \C \; |\; a <\mathrm{Im}\, z < b\}. 
\end{equation}
For $a,b \in \R$, let $\omega_{a,b}:\R \to (0,\infty)$ be given by 
\begin{equation}\label{eq-weightdef}
\omega_{a,b}(t) =\frac{e^{-4\pi at}-e^{-4\pi bt}}{4\pi t}, \quad \text{for all }t\in \R.
\end{equation}
Let $L^2(\R,\omega_{a,b})$ be the Hilbert space of measurable functions $f$ on $\R$ such that 
\begin{align*}
 \norm{f}_{L^2(\mathbb{R},\omega_{a,b})}^2:=\int_{\mathbb{R}} |f(t)|^2\omega_{a,b}(t)\d t<\infty.
\end{align*}

\begin{theorem}[Paley-Wiener theorem for Bergman space of the strip, cf. \cite{koranyi1962bergman}]\label{PWonedimensiontrans} The mapping $T_S:L^2(\R,\omega_{a,b})\to A^2(S(a,b))$ given by
\begin{equation}\label{eq-transisometry1}
T_Sf(z) =\int_{\R} f(t)e^{i2\pi zt}\d t, \quad \text{for all } z\in S(a,b)
\end{equation}   
is an isometric isomorphism of Hilbert spaces. Also, for each function $F\in A^2(S(a,b))$ and for $t\in \R$ the inverse $T^{-1}_S:  A^2(S(a,b))\to  L^2(\R,\omega_{a,b})$ of $T_S$ is represented by  
\begin{equation}\label{eq-invpwtrans}
    T^{-1}_S F(t) = \int_{\R} F(x+ic)e^{2\pi ct}e^{-i2\pi xt}\d x,
\end{equation}
for any $c\in (a,b)$.
\end{theorem}
\begin{proof}

For all $f\in L^2(\R,\omega_{a,b})$ and all $z\in S(a,b)$, by the Cauchy-Schwarz inequality we have
\begin{equation}\label{eq-TSconv}
\int_{\R} \abs{f(t)e^{i2\pi zt}}\d t \leq \left(\int_{\R} \abs{f(t)}^2\omega_{a,b}(t)\right)^{1/2}\left( \int_{\R} \abs{\frac{e^{i2\pi zt}}{\omega_{a,b}(t)}}^2\omega_{a,b}(t)\d t\right)^{1/2}.
\end{equation}
Using the explicit expression for $\omega_{a,b}$ we see  see  that whenever $a<y<b$, i.e., when $z\in S(a,b)$, the second integral on the right hand side 
of  \eqref{eq-transisometry1} converges.  Further, since the 
 integrand in \eqref{eq-transisometry1} is holomorphic in $z$, it is not difficult to verify that $T_Sf$ is holomorphic on the strip $S(a,b)$.
We now compute the $A^2$ norm of $T_Sf$:
\begin{align*}
\norm{T_Sf}^2_{A^2(S(a,b))}     &= \int_a^b \Vert T_Sf(\cdot+iy)\Vert^2_{L^2(\mathbb{R})}\d y \nonumber\\
                                &= \int_a^b \Vert e^{-2\pi y(\cdot)}f\Vert^2_{L^2(\mathbb{R})}\d y \quad \textrm{(By Plancherel's Theorem)} \nonumber\\
                                &= \int_{\mathbb{R}} \abs{f(t)}^2 \int_a^b e^{-4\pi y t}\d y\d t              =\norm{f}^2_{L^2(\mathbb{R};\omega_{a,b})}. 
\end{align*}
Note that the value of the inner integral is $\omega_{a,b}(t)$ given by \eqref{eq-weightdef} and this observation gives the last equality. This calculation shows that $T_S$ is an isometry of $L^2(\R,\omega_{a,b})$ with a subspace of $A^2(S(a,b))$.


 For a function $F\in A^2(S(a,b))$ and $c \in (a,b)$, let $F_c: \R \to \C$ be given by $F_c(x) = F(x+ic)$. To show that $T_S$ is surjective, we show that for each $F$ in $A^2(S(a,b))$, we may choose a $c\in (a,b)$ and obtain for $z\in S(a,b)$
\[
F(z)=T_S(\widehat{F}_ce^{2\pi c (\cdot)})(z)= \int_{\R} \widehat{F}_c(t)e^{2\pi ct}e^{i2\pi zt}\d t = \int_{\R} \widehat{F}_c(t)e^{i2\pi (z-ic)t}\d t , 
\]
where $\widehat{F}_c$ is the $L^2$ Fourier transform of the function $F_c$. To see this, we note that 
\[
\norm{F}_{A^2(S(a,b))}^2 = \int_a^b \norm{F_y}_{L^2(\R)}^2 \d y < \infty,
\]
which shows that $F_y \in L^2(\R)$ for almost all $y \in (a,b)$. We choose a $c\in (a,b)$ such that $F_c \in L^2(\R)$ and let
\begin{equation}\label{eq-G}
G(z)= T_S(\widehat{F}_ce^{i2\pi c(\cdot)})(z)= \int_{\mathbb{R}}\widehat{F}_c( t)e^{i2\pi (z-ic) t}\,d t.
\end{equation} 

First suppose that $\widehat{F}_c$ is compactly supported. Since $F_c \in L^2(\mathbb{R})$, by Plancherel's theorem it follows that $\widehat{F}_c \in L^2(\mathbb{R})$. 
A routine application of Morera's theorem or differentiation under the integral sign shows that the function $G$ is holomorphic in $S(a,b)$.

By the Fourier inversion formula, $G_c=F_c$ on the line $y=c$, and, by the identity theorem for holomorphic functions, we must have $G=F$ on $S(a,b)$. 
Thus we see that
\begin{equation}\label{eq-transisometrycom}
F=T_S(\widehat{F}_ce^{2\pi c (\cdot)}) \quad \text{and} \quad 
\norm{\widehat{F}_ce^{2\pi c (\cdot)}}_{L^2(\mathbb{R};\omega_{a,b})}=\norm{ F}_{A^2(S(a,b))},
\end{equation}
where we used Plancherel's theorem to arrive at the second equality in \eqref{eq-transisometrycom}.

When $\widehat{F_c}$ is not compactly supported, for each $\epsilon >0$, we construct a function $G^{\epsilon} \in A^2(S(a,b))$ with the following properties:
\begin{enumerate}
\item The measurable function $t \mapsto e^{2\pi yt}\widehat{G^{\epsilon}}_y(t)$ is independent of the choice of $y \in (a,b)$.
\item For each $y\in (a,b)$, the function $\widehat{G^{\epsilon}_y}\to \widehat{F_y}$ uniformly on compact subsets of $\mathbb{R}$ as $\epsilon \to 0$.
\end{enumerate}
Let $\epsilon_j>0$ be a decreasing sequence of numbers such that $\epsilon_j\to 0$ as $j\to \infty$. Then, by choosing a $c \in (a,b)$ such that $G^{\epsilon_j} \in L^2(\R)$ for each $j$ it follows from property (1) of $G^{\epsilon}$ that for almost all $t\in \R$, we get
\[
e^{2\pi ct}\widehat{G^{\epsilon_j}_c}(t)=e^{2\pi yt}\widehat{G^{\epsilon_j}_y}(t),
\]for all $y \in (a,b)$. 
Letting $\epsilon_j \searrow 0$, and using property (2) of $G^{\epsilon}$, we get
\begin{align*}
\widehat{F_y}( t) &= e^{-2\pi(y-c) t}\widehat{F_c}( t), \quad \text{for almost all} \:  t \in \mathbb{R}
\end{align*}
Therefore, denoting the inverse $L^2$ Fourier transform by $\mathcal{F}^{-1}$ we have
\[ F_y= \mathcal{F}^{-1}\left(\widehat{F_y}\right)=  \mathcal{F}^{-1}\left(e^{-2\pi (y-c)(\cdot)}\widehat{F_c}\right)=
T_S\left(e^{2\pi c(\cdot)}\widehat{F_c}\right).
\]
This shows that $F$ is in the range of $T_S$, and  \eqref{eq-invpwtrans} holds for some $c$.

We now construct for each $\epsilon>0$,  a function $G^{\epsilon}$ satisfying (1) and (2) above. Let $\phi \in \mathscr{C}^{\infty}
(\mathbb{R})$ be such that $\widehat{\phi}\in \mathscr{C}^{\infty}_c(\mathbb{R})$ and $\widehat{\phi} \equiv 1$ on the interval $[-1,1]$. Let $\phi_{\epsilon}(x)=\epsilon^{-1}\phi(x/\epsilon)$, so that $\widehat{\phi_{\epsilon}}( t)=\widehat{\phi}(\epsilon t)$. Note that $\int_{\mathbb{R}} \abs{\phi_{\epsilon}(x)}\,dx=\Vert\phi\Vert_{L^1(\mathbb{R})}$.
Now for each $\epsilon >0$, and $z\in S(a,b)$ let $G^{\epsilon}$ be given by 
\[
G^{\epsilon}(z)=G^{\epsilon}_y(x)=\left(\phi_{\epsilon}\ast F_y\right)(x).
\]

It is clear that property (2) above is satisfied. 
It remains to be shown that for each $\epsilon >0$, property (1) holds and that $G^{\epsilon}\in A^2(S(a,b))$.

First we show that $G^{\epsilon}$ is holomorphic in $S(a,b)$ for every $\epsilon>0$. For $N>0$, let $\psi_{N} \in \mathscr{C}_c^{\infty}(\R)$ be a function such that $\psi_N \equiv 1$ on $[-N,N]$ and $0\leq \psi_N <1$ outside $[-N,N]$. For $z \in S(a,b)$ let $G^{\epsilon,N}$ be given by
\[
G^{\epsilon,N}(z)= G^{\epsilon, N}_y(x) = \left(\psi_N\phi_{\epsilon}\ast F_y\right)(x)=\int_{\R} F_y(x-t)\psi_N(t)\phi_{\epsilon}(t)\d t=\int_{\R} F(z-t)\psi_N(t)\phi_{\epsilon}(t) \d t. 
\]
Since the function $(z,t)\mapsto \frac{\partial F}{\partial z}(z-t)\psi_N(t)\phi_{\epsilon}(t)$ is continuous on $S(a,b)\times \R$, we may differentiate under the integral sign to see that $G^{\epsilon,N}$ is holomorphic in $S(a,b)$. Note that by Young's inequality for convolution, we get the following estimate 
\begin{align}
\norm{\psi_N\phi_{\epsilon}\ast F_y}_{L^2(\mathbb{R})} \leq \norm{\psi_N\phi_{\epsilon}}_{L^1(\mathbb{R})}\norm{F_y}_{L^2(\mathbb{R})} & \leq \norm{\phi_{\epsilon}}_{L^1(\R)}\norm{F_y}_{L^2(\R)} = \norm{\phi}_{L^1(\mathbb{R})}\norm{F_y}_{L^2(\mathbb{R})}. \label{eq-youngconv}
\end{align}
Now we estimate the $L^2$ norm of $G^{\epsilon,N}$:
\begin{align*}\
\norm{G^{\epsilon,N}}^2_{A^2(S(a,b))}                &= \int_a^b \norm{G^{\epsilon,N}_y}^2_{L^2(\mathbb{R})}\d y = \int_a^b \norm{\psi_N\phi_{\epsilon}\ast F_y}_{L^2(\mathbb{R})}^2\d y  \\
                                                       &\leq \norm{\phi}^2_{L^1(\mathbb{R})} \int_a^b \norm{F_y}^2_{L^2(\mathbb{R})}\d y  = \Vert\phi\Vert^2_{L^1(\mathbb{R})} \Vert F\Vert^2_{A^2(S(a,b))} <\infty \quad \text{(By \eqref{eq-youngconv})}. 
\end{align*}
This shows that $G^{\epsilon,N}\in A^2(S(a,b))$ for all $\epsilon, N >0$.  
Note that $G^{\epsilon}_y-G^{\epsilon,N}_y=(\phi_{\epsilon}-\psi_N\phi_{\epsilon})\ast F_y$ and using Young's inequality for convolution as we did to estimate the norm of $G^{\epsilon,N}$ above gives 
\[
\norm{G^{\epsilon}-G^{\epsilon,N}}^2_{L^2(S(a,b))} \leq \norm{(1-\psi_N)\phi_{\epsilon}}_{L^1(\R)}^2 \norm{F}^2_{A^2(S(A,b))}.
\]
Thus for each $\epsilon>0$, we see that $G^{\epsilon,N}\to G^{\epsilon}$ in $L^2(S(a,b))$ since $\norm{(1-\psi_N)\phi_{\epsilon}}_{L^1(\R)}\to 0$ as $N\to \infty$ by the dominated convergence theorem. Consequently, $G^{\epsilon} \in A^2(S(a,b))$ since $A^2(S(a,b))$ is closed in $L^2(S(a,b))$. 

Since 
$
\norm{G^{\epsilon}}_{A^2(S(a,b))}^2=\int_a^b \norm{G^{\epsilon}_y}_{L^2(\R)}^2 \d y  <\infty,
$
we may choose a $c \in (a,b)$ such that $\norm{G^{\epsilon}_c}_{L^2(\R)} < \infty$. Since $\widehat{G^{\epsilon}_c}=\widehat{\phi_{\epsilon}}\widehat{F}_c$, the function $\widehat{G^{\epsilon}_c}$ is compactly supported and \eqref{eq-transisometrycom} holds for such a $G^{\epsilon}$ in $A^2(S(a,b))$. We obtain by our choice of $c$, 
\[
G^{\epsilon}(z)=T\left(\widehat{G^{\epsilon}_c}e^{2\pi c(\cdot)}\right)(z)=\int_{\mathbb{R}} \widehat{G^{\epsilon}_c}(t)e^{2\pi ct}e^{i2\pi zt}\d t, 
\]
and
$
\Vert \widehat{G^{\epsilon}_c}e^{2\pi c(\cdot)}\Vert_{L^2(\mathbb{R})} =\Vert G^{\epsilon}\Vert_{A^2(S(a,b))}.
$
Thus $e^{2\pi c(\cdot)}\widehat{G^{\epsilon}_c}\in L^2(\mathbb{R})$, and since $\widehat{G^{\epsilon}_c}$ is compactly supported $\widehat{G^{\epsilon}_c} e^{-2\pi(y-c)(\cdot)}$ belongs to $L^1(\mathbb{R})$, for all $y\in (a,b)$. Thus, an application of the Riemann-Lebesgue lemma shows that 
\[
    \lim_{|N|\to \infty} G^{\epsilon}(N+iy)=\lim_{|N|\to \infty} G^{\epsilon}_y(N)=\lim_{|N|\to \infty} \int_{\mathbb{R}} \widehat{G^{\epsilon}_c}(t)e^{-2\pi(y-c)t} e^{i2\pi Nt}\d t =0,
\]
for all $y\in (a,b)$.

We now show that property  (1) holds. For $N>0$, let $S_N$ be the rectangle with vertices $-N+ic, N+ic, N+iy$, and $-N+iy$ where $y\in (a,b)$.

Let $\partial S_N$ be the boundary of $S_N$ oriented clockwise. Then for all $t\in \mathbb{R}$, by Cauchy's theorem applied to the function $z\mapsto G^{\epsilon}(z)e^{-i2\pi zt}$ along the contour $\partial S_N$ oriented clockwise we have:
\begin{align}
    & \int_{-N}^N G^{\epsilon}(x+ic)e^{2\pi ct} e^{-i2\pi xt}\d x + \int_c^y G^{\epsilon}(N+i\eta)e^{2\pi \eta t} e^{-i2\pi Nt}\d\eta  \nonumber\\
    & +\int_N^{-N} G^{\epsilon}(x+iy)e^{2\pi yt} e^{-i2\pi xt}\d x + \int_y^c G^{\epsilon}(-N+i\eta)e^{2\pi \eta t} e^{i2\pi Nt}\d \eta =0 \label{Cauchytheorem}     
\end{align}

Since $\lim_{N\to\infty} G^{\epsilon}(N+i\eta)=0$, for sufficiently large $N$, the function $\eta\mapsto G^{\epsilon}(N+i\eta)e^{2\pi \eta t}e^{-i2\pi Nt}$ 
belongs to $L^1([c,y])$ and  by the dominated convergence theorem we have
\begin{align*}
\lim_{N\to \infty}\int_c^y G^{\epsilon}(N+i\eta)e^{2\pi \eta t} e^{-i2\pi Nt}\d\eta  =0.
\end{align*}
Similarly $\lim_{N\to \infty}\int_y^c G^{\epsilon}(-N+i\eta)e^{2\pi \eta t} e^{i2\pi Nt}\d\eta=0.$
Thus, by letting $N\to \infty$ in \eqref{Cauchytheorem} we see that two terms vanish in the limit  
and from the remaining terms we get
\begin{align*}
    e^{2\pi ct} \widehat{G^{\epsilon}_c}(t)- e^{2\pi yt} \widehat{G^{\epsilon}_y}(t)=0.
\end{align*}
Thus, property (1) holds as claimed.

This completes the proof, except that we have \eqref{eq-invpwtrans} only for some $c$. However, an argument using integration along the sides of the rectangle $S_N$ and letting $N\to\infty$ (as was done for $G^\epsilon$ above) shows that the integral is actually independent of $c$. 
\end{proof}

\subsection{Proof of Theorem \ref{PWTranslation}}
Let $\mathcal{L}_{\mathrm{Trans}}(p)$ be the Hilbert space of measurable functions $g:(0,\infty)\times \C^n\to \C$ such that
\begin{equation}\label{eq-Lpnorm}
 \norm{g}_{\mathcal{L}_{\mathrm{Trans}}(p)}^2:=\int_{\C^n} \int_0^{\infty}\abs{g(t,w)}^2 \frac{e^{-4\pi p(w)t}}{4\pi t} \d V(w) \d t < \infty.
\end{equation}
Recall from \eqref{eq-hpnorm} and \eqref{eq-hphol} that $\mathcal{H}_p$ is the closed subspace of $\mathcal{L}_{\mathrm{Trans}}(p)$ consisting of functions $g \in \mathcal{L}_{\mathrm{Trans}}(p)$ such that  
$\frac{\partial g}{\partial \overline{w}_j}=0 $   in the sense of distributions, for $ 1\leq j\leq n.$

We make use of the following fact when integrating over $\mathcal{U}_p$: for each $w\in \C^n$ such that $(z,w) \in \mathcal{U}_p=\{ (z,w)\in \C\times \C^n\;|\; \textrm{Im}\,z>p(w)\}$, we have $z\in S(p(w),\infty)$. where we set $S(p(w),\infty)=\{z\in \C\;|\; \textrm{Im}\, z>p(w)\}$ as in \eqref{eq-stripdef}.

To prove Theorem~\ref{PWTranslation}, we first show that the map $T_S$ is an isometry from $\mathcal{H}_p$ into $L^2(\mathcal{U}_p)$. 
Since $f\in \mathcal{H}_p$ and the norm in $\mathcal{H}_p$ is given by \eqref{eq-Lpnorm}, it follows by Fubini's theorem that for almost all $w\in \C^n$, we have
\[
\norm{f(\cdot,w)}_{L^2(\R,\omega_{p(w),\infty})}^2=\int_0^{\infty} \abs{f(t,w)}^2 \frac{e^{-4\pi p(w)t}}{4\pi t}\d t <\infty.
\]
For such a $w \in \C^n$, by Theorem \ref{PWonedimensiontrans}, the function $T_Sf(\cdot,w)$ given by
\[
T_Sf(z,w) = \int_0^{\infty} f(t,w) e^{i2\pi zt}\d t, \quad \text{for all } z\in S(p(w),\infty)
\] 
is well defined and 
\begin{equation}\label{eq-transisometryep}
\norm{T_Sf(\cdot,w)}_{A^2(S(p(w),\infty))}^2= \int_{S(p(w),\infty)} \abs{T_Sf(z,w)}^2 \d V(z) = \int_0^{\infty} \abs{f(t,w)}^2 \frac{e^{-4\pi p(w)t}}{4\pi t}\d t.
\end{equation}
Integrating equation \eqref{eq-transisometryep} over $\C^n$, we get
\[
\int_{\C^n} \int_{S(p(w),\infty)} \abs{T_Sf(z,w)}^2 \d V(z) \d V(w) = \int_{\C^n} \int_0^{\infty} \abs{f(t,w)}^2 \frac{e^{-4\pi p(w)t}}{4\pi t}\d t \d V(w) \]
i.e., $\norm{T_Sf}_{L^2(\mathcal{U}_p)}^2    = \norm{f}_{\mathcal{H}_p}^2,$ which shows that $T_S$ is an isometry into $L^2(\mathcal{U}_p)$.

We now show that the image of $T_S$ lies in $A^2(\mathcal{U}_p)$. Let $\varphi_N \in \mathscr{C}_c^{\infty}([0,\infty))$ be such that $\varphi_N\equiv 1$ on $[0,N]$ and $0\leq \varphi \leq 1$ on $[0,\infty)$. 
For each positive integer $N$, we have  $\varphi_N f\in \mathcal{H}_p$ since it is clear that $\norm{\varphi_Nf}_{\mathcal{H}_p} \leq \norm{f}_{\mathcal{H}_p}$ and 
$
\frac{\partial (\varphi_N f)}{\partial \overline{w}_j} =0$ in the sense of distributions for $ j=1,\dots, n.$
Thus, for almost all $w\in \C^n$, 
\[
\norm{\varphi_Nf(\cdot,w)}_{L^2(\R,\omega_{p(w),\infty})}\leq \norm{f(\cdot,w)}_{L^2(\R,\omega_{p(w),\infty})}  <\infty. 
\]
For such a $w$ by Theorem \ref{PWonedimensiontrans}, the function $T_S(\varphi_Nf)(\cdot,w)$ belongs to $A^2(S(p(w),\infty))$ which shows that
$T_S(\varphi_N f)$ is holomorphic in the variable $z$. An easily justified 
differentiation under the integral sign (in the sense of distributions) shows
$T_S(\varphi_Nf)$ is holomorphic in the variable $w_j$ for $1\leq j \leq n$.
Thus, by Hartogs' theorem on separate analyticity we see that $T_S(\varphi_Nf) \in \mathcal{O}(\mathcal{U}_p)$ and consequently $T_S(\varphi_Nf)\in A^2(\mathcal{U}_p)$. By the dominated convergence theorem we obtain $\norm{f-\varphi_Nf}_{\mathcal{H}_p} \to 0$ as $N\to \infty$ and since $T_S$ is an isometry it follows that $T_S(\varphi_Nf) \to T_Sf$ in $L^2(\mathcal{U}_p)$. Since $A^2(\mathcal{U}_p)$ is closed in $L^2(\mathcal{U}_p)$ it follows that $T_Sf\in A^2(\mathcal{U}_p)$.

We now show that the map $T_S$ is surjective. Suppose $F\in A^2(\mathcal{U}_p)$. 
We have
\[
\norm{F}_{A^2(\mathcal{U}_p)}^2=\int_{\C^n} \int_{S(p(w),\infty)} \abs{F(z,w)}^2 \d V(z) \d V(w) < \infty.
\]
By Fubini's theorem, we see that the inner integral above is finite for almost all $w\in \C^n$. 
Thus, by Theorem \ref{PWonedimensiontrans}, for such a $w\in \C^n$, there is a measurable function $f(\cdot, w) \in L^2(\R,\omega_{p(w),\infty})$ such that
\[
F(z,w) = \int_0^{\infty} f(t,w)e^{i2\pi zt}\d t
\]
and 
\begin{align}
\norm{F(\cdot,w)}^2_{A^2(S(p(w),\infty))} &=\int_{S(p(w),\infty)} \abs{F(z,w)}^2 \d V(z) 
 		&= \int_0^{\infty} \abs{f(t,w)}^2 \frac{e^{-4\pi p(w)t}}{4\pi t}\d t \label{eq-transisometryepsurj}.
\end{align}
Using notation of Theorem \ref{PWonedimensiontrans},  $f(\cdot,w)$ is given by  
\begin{equation}\label{eq-tsinverse}
f(t,w) = T^{-1}_SF(\cdot,w)=\int_{\R} F(x+ic,w)e^{i2\pi ct}e^{-i2\pi xt}\d x,
\end{equation}
for any $c >p(w)$. 
Integrating \eqref{eq-transisometryepsurj} over $\C^n$ with respect to $w$, we get
\begin{align}
\norm{F}_{A^2(\mathcal{U}_p)}^2 = \int_{\C^n} \norm{F(\cdot,w)}^2_{A^2(S(p(w),\infty))}  
								 \int_{\C^n} \int_0^{\infty} \abs{f(t,w)}^2 \frac{e^{-4\pi p(w)t}}{4\pi t}\d t \nonumber = \norm{f}^2_{\mathcal{L}_{\mathrm{Trans}}(p)}.\label{eq-surjectivityisometrytrans}
\end{align}
This shows that $f\in \mathcal{L}_{\mathrm{Trans}}(p)$. Note that the integrand in the formula \eqref{eq-tsinverse} is holomorphic in the variable $w$. An easily justified computation of the distributional derivative then shows that $(t,w)\mapsto f(t,w)$ is holomorphic in the variable $w$ and hence $f\in \mathcal{H}_p$. This completes the proof.

\section{Fourier representation with respect to the scaling group}\label{sec-scal}

To study Fourier representations with respect to the scaling group we introduce $\mathcal{X}_p$, another Hilbert space isometrically isomorphic to $A^2(\mathcal{U}_p)$. 
For $s\in (-1,1)$ and $t\in \R$ let $\lambda$ be given by
\begin{equation}\label{eq-lambdadef}
\lambda(s,t) =
\begin{cases}
\dfrac{1}{4\pi t}\left( e^{-4\pi \sin^{-1}(s) t}-e^{-4\pi (\pi-\sin^{-1}(s)) t}\right), \; \text{if } t\neq 0\\
\pi - 2\sin^{-1}s, \quad \text{if } t=0. 
\end{cases}
\end{equation}

Recall from \eqref{eq-Bpdef} that we denote by $\mathbb{B}_p=\{ w\in \C^n\;|\; p(w)<1\}$. Let $\mathcal{X}_p$ be the Hilbert space consisting of measurable functions $f$ on $\R\times \mathbb{B}_p$ such that 
\begin{equation} \label{eq-xpnorm}
\norm{f}_{\mathcal{X}_p}^2:=\int_{\R}\int_{\mathbb{B}_p}\abs{f(t,\zeta)}^2 \lambda(p(\zeta),t)\d V(\zeta)\d t < \infty,
\end{equation}
and such that
\begin{equation}\label{eq-xphol}
\frac{\partial f}{\partial \overline{\zeta}_j}=0 \, \text{ in the sense of distributions, } 1\leq j\leq n,
\end{equation}
where $\zeta_1,\dots, \zeta_n$ are the natural co-ordinates of $\C^n$. That is, $\mathcal{X}_p$ is the space of measurable functions on $\R\times \mathbb{B}_p$ which are holomorphic in the variable $\zeta \in \mathbb{B}_p$  and square integrable with respect to the weight $(t,\zeta) \mapsto \lambda(p(\zeta),t)$ on $\R\times \mathbb{B}_p$, where recall that $p$ is the weighted homogeneous balanced polynomial defining the domains $\mathcal{U}_p =\{\mathrm{Im}\, z> p(w)\}\subset \C^{n+1}$ and $\mathbb{B}_p=\{p<1\}\subset \C^n$.

Let $\gamma\in \C$, and suppose that $\gamma$ does not lie on the negative real axis. Let the map $\widehat{\rho}_{\gamma}:\C^n\to \C^n$ be given by 
\begin{equation}\label{eq-rhohat}
\widehat{\rho}_{\gamma}(w_1,\dots,w_n) = \left(\gamma^{1/2m_1}w_1,\dots,\gamma^{1/2m_n}w_n\right),
\end{equation}  
where the powers of $\gamma$ are defined using the principal branch of the logarithm, i.e., the one on the complex plane slit along the negative real axis which coincides with the real valued natural logarithm on the positive real line. We can now state 
the representation theorem with respect to the scaling group, which involves an 
inverse Mellin transform:

\begin{theorem}\label{PWDilation}
Let $1/\mu = 1/m_1+\cdots+1/m_n$. For $(z,w) \in \mathcal{U}_p$, the map $T_V:\mathcal{X}_p\to A^2(\mathcal{U}_p)$ given by
\begin{align}
T_Vg(z,w) &= \int_{\R} g\left(t,\widehat{\rho}_{1/z}(w)\right) \frac{z^{i2\pi t}}{z^{1+1/2\mu}}\d t \label{eq-pwdilation} =\int_{\R} g\left(t,\frac{w_1}{z^{1/2m_1}},\cdots,\frac{w_n}{z^{1/2m_n}} \right) \frac{z^{i2\pi t}}{z^{1+1/2\mu}}\d t 
\end{align}
is an isometric isomorphism of Hilbert spaces.
\end{theorem}
 \subsection{The case of the sector} As a step towards the proof of 
 Theorem~\ref{PWDilation}, we consider Mellin integral representations of 
 the Bergman space of a sector in $\C$, since such sectors (which are not polynomial half spaces) are the simplest domains admitting scaling automorphisms. For $0\leq a <b\leq \pi$, let 
\begin{equation}\label{eq-sectordef}
V(a,b) = \{ z\in \C\; |\; a <\arg \,z<b \}.
\end{equation}

 \begin{theorem}[Paley-Wiener Theorem for the Bergman Space of the Sector]\label{PWonedimensiondilation} 
The mapping $T_V:L^2(\mathbb{R},\omega_{a,b})\to A^2(V(a,b))$ given by
\begin{equation}\label{eq-dilationisometry1}
T_Vf(z)=\int\limits_{\mathbb{R}} f(t)\dfrac{z^{i2\pi t}}{z}\d t, \quad \text{for all } z\in V(a,b) 
\end{equation}
is an isometric isomorphism of Hilbert spaces. 
\end{theorem}
Here $\omega_{a,b}$ is the weight in  \eqref{eq-weightdef}, which occurs in the Fourier representation Theorem~\ref{PWonedimensiontrans} in the strip $S(a,b)=\{z\in \C|a< {\rm Im} z<b\}.$ 
\begin{proof}

The mapping $\Phi: V(a,b) \to S(a,b)$ given by $\Phi(z)=\log |z|+ i \arg z$
is a conformal equivalence, 
where  the argument is chosen such that $a<\arg z<b$ for all $z\in V(a,b)$. This conformal equivalence establishes an isometric isomorphism $\Phi^*: A^2(S(a,b))\to A^2(V(a,b)) $ given by
\begin{equation}\label{eq-onedimenscal}
  \Phi^*(F)=(F\circ \Phi)\cdot(\Phi').
\end{equation}

By Theorem \ref{PWonedimensiontrans}, we know that $T_S:L^2(\mathbb{R},\omega_{a,b}) \to A^2(S(a,b))$ given by 
$
T_S(f)=\int_{\mathbb{R}} f(t)e^{i2\pi zt}\d t
$ is an isometric isomorphism.  Thus, $\Phi^*\circ T_S$ is an isometric isomorphism of $L^2(\mathbb{R},\omega_{a,b})$ with $A^2(V(a,b))$, which we compute below and show equal to $T_V$:
\begin{align}
    \Phi^*\circ T_S(f)(z)  &= \Phi^*\left( \int_{\mathbb{R}} f(t)e^{i2\pi zt}\d t\right) = \left(\int_{\mathbb{R}} f(t)e^{i2\pi \Phi(z)t}\d t\right)\cdot \Phi'(z) = \int\limits_{\mathbb{R}} f(t)\dfrac{z^{i2\pi t}}{z}\d t, \label{eq-PWscalformula}
\end{align}
where we used the fact that $\Phi'(z)=1/z$.
\end{proof}

\subsection{Remark.}\label{cor-PWonedimensionscal}
Let $\left(\Phi^*\right)^{-1}: A^2(V(a,b)) \to A^2(S(a,b))$ be the inverse of the isometric isomorphism $\Phi^*$, where $\Phi^*$ is as in \eqref{eq-onedimenscal}. It follows from equation \eqref{eq-PWscalformula} that for $F\in A^2(S(a,b))$,
\[
T_V^{-1}F(t) = (T_S^{-1}\circ \left(\Phi^*\right)^{-1})F(t) = T^{-1}_S\left(\left(\Phi^*\right)^{-1}F\right)(t)= \int_{\R} \left(\Phi^*\right)^{-1}F(x+ic)e^{2\pi ct}e^{-i2\pi xt} \d x,
\]
where the last equality follows from Theorem \ref{PWonedimensiontrans}. Since $T^{-1}_S$ and $\left(\Phi^*\right)^{-1}$ are isometries, we obtain
\[
\norm{T_V^{-1}F}_{L^2(\R,\omega_{a,b})}= \norm{T^{-1}_S\circ \left(\Phi^*\right)^{-1}F}_{L^2(\R,\omega_{a,b})} = \norm{\left(\Phi^*\right)^{-1}F}_{A^2(S(a,b))}= \norm{F}_{A^2(V(a,b))}.
\]

\subsection{The domain $\mathcal{C}_p$} In preparation of the proof of
Theorem~\ref{PWDilation}, we collect a few results that will be needed.
Introduce a domain $\mathcal{C}_p \subset \C^{n+1}$ biholomorphically equivalent to $\mathcal{U}_p$, by setting 
\begin{equation}\label{eq-Cpdef}
\mathcal{C}_p = \left\{ (\gamma,\zeta) \in \C \times \C^{n} \; | \; \textmd{Im} \; \gamma > p(\zeta)\abs{\gamma}\right\}.
\end{equation}
 If $(z,w) \in \mathcal{U}_p$, we have $\textrm{Im}\; z> p(w) \geq 0$, so the principal branch of the logarithm is defined on 
$\{ z\in \C \; | \; \textrm{Im}\; z > 0\}$ independently of $w\in \C^n$. Define the powers $z^{1/2m_1},\cdots,z^{1/2m_n}$ with respect to this branch. With this understanding, we have the following:

\begin{lemma}\label{lem-Cp}
The map $\Psi: \mathcal{U}_p\to \mathcal{C}_p$ given by
\[
\Psi(z,w) = \left(z, \frac{w_1}{z^{1/2m_1}},\cdots,\frac{w_n}{z^{1/2m_n}}\right)=\left(z, \widehat{\rho}_{1/z}(w)\right), \quad \text{for all } (z,w) \in \mathcal{U}_p,
\]
where $\widehat{\rho}_{1/z}$ is as in \eqref{eq-rhohat} is a biholomorphic equivalence.
\end{lemma} 

\begin{proof}
First we note that if $\gamma \in \C$ does not lie on the negative real axis, then by defining the powers $\gamma^{1/2m_1},\dots, \gamma^{1/2m_n}$ by with respect to the principal branch of logarithm and performing an easy computation we see that $p(\widehat{\rho}_{\gamma}(w))=\abs{\gamma}p(w)$ for all $w\in \C^n$.  Now, we show that $\Psi(z,w) \in \mathcal{C}_p$. 
Since $(z,w) \in \mathcal{U}_p$ we have $\textrm{Im}\; z> p(w)$.
We then get
\[
 \textmd{Im} \; z >p(w) = p\left(\widehat{\rho}_{z} (\widehat{\rho}_{1/z}(w))\right)= p\left(\widehat{\rho}_{1/z}(w)\right) \abs{z}. 
\]
Thus $\Psi(z,w) \in \mathcal{C}_p$.
If $(z,w), (z',w') \in \mathcal{U}_p$ then by solving the equation $\Psi(z,w) = \Psi(z',w')$ it is easy to see that $\Psi$ is one-to-one. 

To see that $\Psi$ is onto, note that for $(\gamma,\zeta) \in \mathcal{C}_p$, we have $p(\zeta)\abs{\gamma} < \textrm{Im}\; \gamma$, which gives
\[
p(\zeta) < \frac{\textrm{Im}\; \gamma}{\abs{\gamma}} < 1,
\]
Thus, for all $\zeta \in \C^n$ such that $(\gamma, \zeta) \in \mathcal{C}_p$, we have $p(\zeta)<1$, i.e., $\zeta \in \mathbb{B}_p$, where recall that $\mathbb{B}_p =\{\zeta \in \C^n\; |\; p(\zeta)<1\}$, as in \eqref{eq-Bpdef}. For each $\zeta \in \mathbb{B}_p$, such that $(\gamma, \zeta) \in \mathcal{C}_p$, we have 
\begin{equation*}
    \textmd{Im}\; \gamma > p(\zeta)\abs{\gamma} \; \; \textrm{i.e.} \; \; \gamma \in V_{\zeta} =\left\{ \gamma \in \C \; |\; \sin^{-1} p(\zeta)< \arg \;z <\pi - \sin^{-1}p(\zeta)\right\}.
\end{equation*}
\noindent Note that $V_{\zeta}$ is a subset of the upper half plane, and hence we define the powers $\gamma^{1/2m_1},\cdots,\gamma^{1/2m_n}$ using the principal branch of the logarithm, independent of $\zeta \in \mathbb{B}_p$.
If $(\gamma,\zeta) \in \mathcal{C}_p$, then it is not difficult to see that $(\gamma, \gamma^{1/2m_1} \zeta_1,\dots,\gamma^{1/2m_n} \zeta_n) \in \mathcal{U}_p$ and $\Psi(\gamma, \gamma^{1/2m_1} \zeta_1,\dots,\gamma^{1/2m_n} \zeta_n)=(\gamma,\zeta)$. 
Since $\Psi$ is a bijective holormorphic map, it is a biholomorphism from $\mathcal{U}_p$ to $\mathcal{C}_p$. 
\end{proof}

\subsection{Proof of Theorem \ref{PWDilation}}

Let $\mathcal{L}_{\mathrm{Scal}}(p)$ be the Hilbert space of measurable functions $f:\R\times \mathbb{B}_p\to \C$ such that
\begin{equation}\label{eq-Lspnorm}
\norm{f}_{\mathcal{L}_{\mathrm{Scal}}(p)}^2:=\int_{\R}\int_{\mathbb{B}_p}\abs{f(t,\zeta)}^2 \lambda(p(\zeta),t)\d V(\zeta)\d t < \infty,
\end{equation}
where for $\zeta \in \mathbb{B}_p$ and $t\in \R$, the function $\lambda(p(\zeta),t)$ is given by (see \eqref{eq-lambdadef})
\begin{equation}\label{eq-weightdef2}
\lambda(p(\zeta),t) =
\begin{cases}
\dfrac{1}{4\pi t}\left(e^{-4\pi t\sin^{-1}p(\zeta)}-e^{-4\pi t (\pi-\sin^{-1}p(\zeta))}\right) \; \text{if } t\neq 0 \\
\pi - 2\sin^{-1}(p(\zeta)) \quad \text{if } t=0.
\end{cases}
\end{equation}
Then $\mathcal{X}_p$ is the closed subspace of $\mathcal{L}_{\mathrm{Scal}}(p)$ consisting of functions $f$ such that  
\[
\frac{\partial f}{\partial \overline{\zeta}_j}=0 \, \text{in the sense of distributions, } 1\leq j\leq n.
\]

We first show that the map $\widetilde{T}_V:\mathcal{X}_p\to A^2(\mathcal{C}_p)$ given by 
\begin{equation}\label{eq-ttildev}
\widetilde{T}_Vf(\gamma,\zeta) = \int_{\R}f(t,\zeta) \frac{\gamma^{2\pi it}}{\gamma}\d t, \quad \text{for all } (\gamma,\zeta) \in \mathcal{C}_p, 
\end{equation}
is an isometric isomorphism from $\mathcal{X}_p$ into $A^2(\mathcal{C}_p)$. The composition $\Psi^*\circ \widetilde{T}_V$ is the isometric isomorphism $\widetilde{T}_V$, where $\Psi^*$ is the isometric isomorphism induced by the biholomorphic map of Lemma~\ref{lem-Cp}.
To simplify notation, let
\begin{equation}\label{eq-azetabzeta}
a(\zeta) = \sin^{-1}p(\zeta) \text{  and  } b(\zeta)=\pi -\sin^{-1}p(\zeta).
\end{equation}
As in the proof of Lemma \ref{lem-Cp}, let $V_{\zeta}$ be the planar sector 
\begin{equation}\label{eq-vzeta}
V_{\zeta} = V\left(a(\zeta),b(\zeta)\right) =\{ \gamma \in \C \; |\; \sin^{-1}p(\zeta) < \arg \gamma < \pi -\sin^{-1}p(\zeta)\}.
\end{equation} 
For $f\in \mathcal{X}_p$, it follows from \eqref{eq-Lspnorm} that for almost all $\zeta\in \mathbb{B}_p$, we have
$
\int_{\R} \abs{f(t,\zeta)}^2 \lambda(p(\zeta),t)\d t <\infty.
$
For such a $\zeta \in \mathbb{B}_p$, by Theorem \ref{PWonedimensiondilation}, the function $\widetilde{T}_Vf(\cdot,\zeta)$ given by 
\[
\widetilde{T}_Vf(\gamma,\zeta) = \int_{\R} f(t,\zeta) \frac{\gamma^{2i\pi t}}{\gamma}\d t, \quad \text{for all } \gamma \in V_{\zeta}
\]
is well defined and
\begin{align}
\norm{\widetilde{T}_Vf(\cdot,\zeta)}_{A^2(V_{\zeta})}^2= \int_{V_{\zeta}} \abs{\widetilde{T}_Vf(\gamma,\zeta)}^2 \d V(z) &= \int_{\R} \abs{f(t,\zeta)}^2 \omega_{a(\zeta),b(\zeta)}(t)\d t = \int_{\R} \abs{f(t,\zeta)}^2 \lambda(p(\zeta),t)\d t \label{eq-scalisometryep},
\end{align}
where the  equality $\omega_{a(\zeta),b(\zeta)}(t)= \lambda(p(\zeta),t)$ follows from definition  \eqref{eq-weightdef} of
$\omega_{a,b}$, definition  \eqref{eq-azetabzeta} of $a(\zeta)$ and $b(\zeta)$, 
and definition \eqref{eq-weightdef2} of $\lambda$.  
Integrating equation \eqref{eq-scalisometryep} over $\mathbb{B}_p$, 
and using the fact that if $\zeta \in \mathbb{B}_p$ is such that $(\gamma,\zeta) \in \mathcal{C}_p$ then $\gamma$ belongs to the sector $V_{\zeta}$ given by \eqref{eq-vzeta}, we get
\[
\norm{\widetilde{T}_Vf}_{L^2(\mathcal{C}_p)}^2 = \int_{\mathbb{B}_p} \int_{V_{\zeta}} \abs{\widetilde{T}_Vf(\gamma,\zeta)}^2 \d V(\gamma) \d V(\zeta)
								= \int_{\mathbb{B}_p} \int_{\R} \abs{f(t,w)}^2 \lambda(p(\zeta),t)\d t \d V(\zeta) 
                                = \norm{f}_{\mathcal{X}_p}^2,
\]
which shows that $\widetilde{T}_V$ is an isometry into $L^2(\mathcal{C}_p)$.
From definition \eqref{eq-ttildev}, we see by a routine computation of the distributional derivative that $(\gamma,\zeta)\mapsto \widetilde{T}_Vf(\gamma, \zeta)$ is holomorphic in each of the variables $\gamma, \zeta$ since the integrand is holomorphic in the variables $\gamma$ and $\zeta$. By Hartogs' theorem on separate analyticity we see that $\widetilde{T}_Vf \in \mathcal{O}(\mathcal{C}_p)$ and thus is in $A^2(\mathcal{C}_p)$.

Now we verify that the map $\widetilde{T}_V$ is surjective. For $F\in A^2(\mathcal{C}_p)$ we have  
\[
\norm{F}^2_{A^2(\mathcal{C}_p)}=\int_{\mathbb{B}_p} \int_{V_{\zeta}} \abs{F(\gamma,\zeta)}^2 \d V(\gamma) \d V(\zeta) < \infty,
\]
where $V_{\zeta}$ is as in \eqref{eq-vzeta}.
Thus, for almost all $\zeta \in \mathbb{B}_p$, the inner integral above is finite.
Thus, by Theorem \ref{PWonedimensiondilation}, for such a $\zeta\in \mathbb{B}_p$, there exists an $f(\cdot,\zeta) \in L^2(\R,\omega_{\alpha(\zeta),\beta(\zeta)})$ such that
\[
F(\gamma,\zeta) = \int_{\R} f(t,\zeta)\frac{\gamma^{i2\pi t}}{\gamma}\d t, \quad \text{for all } \gamma \in V_{\zeta}
\]
and 
\begin{equation}\label{eq-scalisometryepsurj}
\int_{V_{\zeta}} \abs{F(\gamma,\zeta)}^2 \d V(\gamma) = \int_{\R} \abs{f(t,\zeta)}^2 \lambda(p(\zeta),t)\d t. 
\end{equation}
Integrating equation \eqref{eq-scalisometryepsurj} over $\mathbb{B}_p$ shows that $\norm{F}_{A^2(\mathcal{C}_p)}=\norm{f}_{\mathcal{L}_{\mathrm{Scal}}(p)}$. Using the notation of Theorem \ref{PWonedimensiondilation}, by Remark \ref{cor-PWonedimensionscal}, the measurable function $f(\cdot,\zeta)$ is given by
\begin{equation}\label{eq-Tvtilinv}
f(t,\zeta) = \int_{\R}\left(\Phi^*\right)^{-1}F(x+ic,\zeta)e^{2\pi ct} e^{-i2\pi xt}\d x:=\widetilde{T}_V^{-1}F(t,\zeta), \; \text{for } t\in \R,
\end{equation}
where $c\in (\sin^{-1}p(\zeta),\pi-\sin^{-1}p(\zeta))$. Observe that the integrand in the above formula is holomorphic in $\zeta$. An easily justified 
computation of the distributional derivative shows that $(t,\zeta)\mapsto f(t, \zeta)$ is holomorphic in the variable $\zeta$ and so $f\in \mathcal{X}_p$. Therefore   $\widetilde{T}_V$ is surjective, and $\widetilde{T}_V:\mathcal{X}_p\to A^2(\mathcal{C}_p)$ is an isometric isomorphism and the map $\widetilde{T}_V^{-1}$ given by equation \eqref{eq-Tvtilinv} is the inverse of $\widetilde{T}_V$. 

The biholomorphism $\Psi$ of Lemma~\ref{lem-Cp}  induces an isomorphism $\Psi^* : A^2(\mathcal{C}_p)\to A^2(\mathcal{U}_p)$, as in \eqref{eq-unitaryrep}.
Then $\Psi^*\circ \widetilde{T}_V$ is an isometric isomorphism from $\mathcal{X}_p$ to $A^2(\mathcal{U}_p)$. We now show that $T_V=\Psi^*\circ \widetilde{T}_V$. For $f\in \mathcal{X}_p$ and $(z,w) \in \mathcal{U}_p$ we have 
\begin{align*}
\Psi^*(\widetilde{T}_Vf)(z,w) &= \widetilde{T}_Vf(\Psi(z,w))\cdot \det \;\Psi'(z,w) \\
&= \left(\int_{\R} f\left(t,\frac{w_1}{z^{1/2m_1}},\dots,\frac{w_n}{z^{1/2m_n}}\right) \frac{z^{2\pi i t}}{z}\d t\right)\frac{1}{z^{1/2\mu}}\\
&= T_Vf(z,w),
\end{align*}
where we used the fact that for $(z,w) \in \mathcal{U}_p$,
$
\textrm{det}\; \Psi'(z,w) = \frac{1}{z^{1/2\mu}}.
$
This completes the proof of the theorem.

\section{Bergman kernels of polynomial half spaces}\label{sec-rkhs}

\subsection{A method of computing Bergman kernels.} Let $H$ be a Hilbert space. By the Riesz representation theorem, for each bounded linear functional $\varphi \in H^*$, there is an element $R_H(\varphi) \in H$ such that 
\begin{equation}\label{eq-Rieszmap}
\varphi(f) = \langle f, R_H(\varphi)\rangle_H, \quad f\in H.
\end{equation}
We call the map $R_H: H^*\to H$ the \textit{Riesz map}, which is a conjugate linear isometry of Hilbert spaces.

Recall that a Hilbert space $H$ of functions on $\Omega$ is called a \textit{reproducing kernel Hilbert space} if, for each $z$ in $\Omega$, the evaluation map $e_z:H\to \C$ given by $e_zf=f(z)$ is a bounded linear functional on $H$.
The function $K:\Omega\times \Omega \to \C$ given by 
\begin{equation}\label{eq-repkerneldef}
K(z,Z)=\overline{R_H(e_z)(Z)},
\end{equation} 
is called the \textit{reproducing kernel} for $H$, where $R_H$ is as in (\ref{eq-Rieszmap}).
It is well-known that 
weighted Bergman spaces are examples of Reproducing kernel Hilbert spaces (cf. \cite{RKHS}). The reproducing kernel of an unweighted Bergman space is the {\em Bergman kernel}, which is one of the 
most important invariants of a complex domain. We will use the following simple proposition 
to investigate the Bergman kernels of polynomial half spaces.

\begin{proposition}\label{lem-repkernelrepresentation}\textmd{(}{cf. \cite[Lemma IX.3.5]{faraut-koranyi}}\textmd{)}
Let $H$ be a reproducing kernel Hilbert space of functions on $\Omega$. Suppose $L$ is another Hilbert and $T: L\to H$ is an isometric isomorphism. Then the function $K: \Omega\times \Omega \to \C$ given by 
\begin{equation}\label{eq-repkernel1}
K(z,Z) = \left\langle R_L(e_Z\circ T),R_L(e_z\circ T)\right\rangle_L,
\end{equation}
is the reproducing kernel for $H$. 
\end{proposition}
\begin{proof}
The map $e_z\circ T:L\to \C$ is bounded linear functional on $L$ since each of $e_z$ and $T$ is bounded and linear. Thus, we have
\begin{align*}
&&Tf(z) &= (e_z\circ T)f = \langle f, R_L(e_z\circ T)\rangle_L  =\langle Tf, TR_L(e_z\circ T)\rangle_H, \end{align*}
where the last equality holds because $T$ is an isometry and $R_L$ is the Riesz map for the Hilbert space $L$ as in \eqref{eq-Rieszmap}.
That is for every $g\in H$ and each $z\in \Omega$, we have
\begin{equation}\label{eq-repkernelrep2}
g(z) = \langle g, TR_L(e_z\circ T) \rangle_H, \quad \textrm{i.e.,} \quad TR_L(e_z\circ T) = R_H(e_z)
\end{equation}
Thus, the reproducing kernel $K$ for the Hilbert space $H$ is given by
\begin{alignat*}{3}
&&K(z,Z) &= \overline{R_H(e_z)(Z)} &&\quad \text{(By \eqref{eq-repkerneldef})} \\
&&		 &=\overline{TR_L(e_z\circ T)(Z)} && \quad \text{(By equation \eqref{eq-repkernelrep2})}\\
&&		 &= \overline{\langle TR_L(e_z\circ T), TR_L(e_Z\circ T)\rangle}_H && \quad \left(\textrm{By equation } (\ref{eq-Rieszmap}) \right) \\
&& 		 &= \langle R_L(e_Z\circ T), R_L(e_z\circ T)\rangle_L. && \quad (T\textrm{ is an isometry}) 
\end{alignat*}
\end{proof}

If we take $T:H\to H$ to be the identity map in Proposition \ref{lem-repkernelrepresentation}, we have the alternative representation of the reproducing kernel
\begin{equation}\label{eq-repkernel}
K(z,Z) = \left\langle R_H(e_Z), R_H(e_z)\right\rangle_H.
\end{equation}
Let $K$ denote the reproducing kernel for the weighted Bergman space $A^2(\Omega,\lambda)$ on $\Omega\subset \C^n$ with weight $\lambda$. Then the relation \eqref{eq-repkernel} takes the form
\begin{equation}\label{lem-transbergmankernelcont}
K(z,Z) = \int_{\Omega} K(z,\zeta) \overline{K(Z,\zeta)} \lambda(\zeta) \d V(\zeta).
\end{equation}

\subsection{A formula for the Bergman kernel of $\mathcal{E}_p$}
 
Let $\mathcal{E}_p$ be a polynomial ellipsoid as in \eqref{eq-Epdef}, where $p$ is a weighted homogeneous balanced polynomial. Recall from \eqref{eq-wpk} that $\mathcal{W}_p(k)$ is the weighted Bergman space on $\mathbb{B}_p=\{p<1\}\subset \C^n$ with respect to the weight $w \mapsto (1-p(w))^{k+1}$. Also recall that $\mathcal{Y}_p$ is the Hilbert space of sequences $a=(a_k)_{k=0}^{\infty}$ such that $a_k \in \mathcal{W}_p(k)$ for each $k\in \N$, with the norm given by \eqref{eq-normYp}. Applying Proposition \ref{lem-repkernelrepresentation} to Theorem \ref{thm-PWCompact} we represent the Bergman kernel of $\mathcal{E}_p$ in terms of the reproducing kernels of the spaces $\mathcal{W}_p(k)$.
 \begin{theorem}\label{thm-Haslingercompact}
For $k\in \N$, let $Y_p(k;\cdot,\cdot)$ be the reproducing kernel for the weighted Bergman space $\mathcal{W}_p(k)$ on the domain $\mathbb{B}_p \subset \C^{n}$. Then for $z,Z \in \C$ and $w, W\in \C^n$ such that $(z,w), (Z,W) \in \mathcal{E}_p$, the Bergman kernel $B$ of $\mathcal{E}_p$ is given by
\[
B(z,w;Z,W) = \sum_{k=0}^{\infty} \frac{k+1}{\pi} Y_p(k;w,W)z^{k}\overline{Z}^k. 
\] 
\end{theorem}

\begin{proof}
We will apply Proposition~\ref{lem-repkernelrepresentation} to the isometric isomorphism  $T: \mathcal{Y}_p \to A^2(\mathcal{E}_p)$ of \eqref{eq-PWCompact} as in Theorem \ref{thm-PWCompact}.
Let $e_{(z,w)}$ be the evaluation functional on $A^2(\mathcal{E}_p)$ at a point $(z,w) \in \mathcal{E}_p$. Then, for $a\in \mathcal{Y}_p$ the map $e_{(z,w)}\circ T: \mathcal{Y}_p \to \C$ is given by
\begin{align}
(e_{(z,w)}\circ T)a = Ta(z,w) &= \sum_{k=0}^{\infty} a_k(w)z^k = \sum_{k=0}^{\infty} \left\langle a_k, \overline{Y_p(k;w,\cdot)}\right\rangle_{\mathcal{W}_p(k)}\cdot z^k \nonumber \\ 
                    &= \sum_{k=0}^{\infty} \frac{\pi}{k+1} \left(\int_{\mathbb{B}_p} a_k(\zeta)Y_p(k;w,\zeta)\frac{(k+1)z^{k}}{\pi}(1-p(\zeta))^{k+1}\d V(\zeta)\right) , \label{eq-compactRiesz1}
 \end{align}
where the last equality follows from the reproducing property of the kernel $Y_p(k;\cdot,\cdot)$. On the other hand, the image $R_{\mathcal{Y}_p}\left(e_{(z,w)}\circ T\right)$  of the functional $e_{(z,w)}\circ T\in \mathcal{Y}_p^*$ under the Riesz map $R_{\mathcal{Y}_p}:\mathcal{Y}_p^*\to \mathcal{Y}_p$ is given for $a\in \mathcal{Y}_p$ by
\begin{align}
(e_{(z,w)}\circ T)a &= \left\langle a,R_{\mathcal{Y}_p}(e_{(z,w)}\circ T)\right\rangle_{\mathcal{Y}_p} 
				= 	\sum_{k=0}^{\infty} \frac{\pi}{k+1}\int_{\mathbb{B}_p} a_k(\zeta)\overline{R_{\mathcal{Y}_p}(e_{(z,w)}\circ T)(k,\zeta)}(1-p(\zeta))^{k+1}\d V(\zeta)  \label{eq-compactRiesz2}
\end{align}
Comparing the two representations \eqref{eq-compactRiesz1} and \eqref{eq-compactRiesz2} of $(e_{(z,w)}\circ T)a$ , we claim that    for every $k\in \N$ and every $\zeta \in \mathbb{B}_p$
\begin{equation}\label{eq-Ypk}
R_{\mathcal{Y}_p}\left(e_{(z,w)}\circ T\right)(k,\zeta) = \overline{Y_p(k;w,\zeta)}\frac{(k+1)\overline{z}^k}{\pi}, \quad (k,\zeta) \in \N\times \mathbb{B}_p.
\end{equation}
Let $\phi: \N\times \mathbb{B}_p\to \C$ be the difference of the two sides of the equation above, i.e., 
\[
\phi(k,\zeta) = R_{\mathcal{Y}_p}\left(e_{(z,w)}\circ T\right)(k,\zeta) - \overline{Y_p(k;w,\zeta)}\frac{(k+1)\overline{z}^k}{\pi}, \quad (k,\zeta) \in \N\times \mathbb{B}_p..
\]
It then follows from equations \eqref{eq-compactRiesz1} and \eqref{eq-compactRiesz2} that the equality 
\begin{equation*}
\sum_{k=0}^{\infty}\frac{\pi}{k+1}\int_{\mathbb{B}_p} a_k(\zeta)\overline{\phi(k,\zeta)}(1-p(\zeta))^{k+1}\d V(\zeta) = 0
\end{equation*}
holds for all $a\in \mathcal{Y}_p$. For a fixed $\ell$ let $a$ be such that $a_k=0$ if $k\neq \ell$ and $a_k=\phi(k,\cdot)$. Since $\phi(k,\cdot) \in \mathcal{W}_p(k)$ (because $Y_p(k;w,\cdot)$ is the reproducing kernel of $\mathcal{W}_p(k)$) we see that $\phi(k,\cdot) \equiv 0$ on $\mathbb{B}_p$. Since this holds for all $k$, the claim \eqref{eq-Ypk} is proved. Applying Proposition \ref{lem-repkernelrepresentation} to the isometric isomorphism $T: \mathcal{Y}_p\to A^2(\mathcal{E}_p)$, and using the representation of $R_{\mathcal{Y}_p}\left(e_{(z,w)}\circ T\right)(k,\zeta)$  given by \eqref{eq-Ypk}, we get 
\begin{align*}
B(z,w;Z,W) &= \left\langle R_{\mathcal{Y}_p}(e_{(Z,W)}\circ T), R_{\mathcal{Y}_p}(e_{(z,w)}\circ T)\right\rangle_{\mathcal{Y}_p} \\
		   &= \sum_{k=0}^{\infty} \frac{\pi}{k+1}\int_{\mathbb{B}_p}\left(Y_p(k;w,\zeta) \frac{(k+1)z^k}{\pi}\right)\left(\overline{Y_p(k;W,\zeta)} \frac{(k+1)\overline{Z}^k}{\pi}\right)(1-p(\zeta))^{k+1} \d V(\zeta) \\
           &= \sum_{k=0}^{\infty} \frac{k+1}{\pi}  \left(\int_{\mathbb{B}_p} Y_p(k;w,\zeta)  \overline{Y_p(k;W,\zeta)} (1-p(\zeta))^{k+1} \d V(\zeta)\right) z^{k}\overline{Z}^k \\
           &= \sum_{k=0}^{\infty} \frac{k+1}{\pi} Y_p(k;w,W)z^{k}\overline{Z}^k,
\end{align*}
where we used \eqref{lem-transbergmankernelcont} to arrive at the last equality.
\end{proof}

\subsection{A formula for the Bergman kernel of $\mathcal{U}_p$}

For $t>0$, let $\mathcal{S}_p(t)$ be the weighted Bergman space on $\C^n$ with weight $w\mapsto e^{-4\pi p(w)t}$, i.e., 
\begin{equation}\label{eq-sptindex}
\mathcal{S}_p(t)= A^2\left( \C^n,e^{-4\pi pt}\right).
\end{equation}
For $p(w)=\abs{w}^2$, i.e., when $\mathcal{U}_p$ is the Siegel upper half-plane, the space $\mathcal{S}_p(t)$ is a \emph{Segal-Bargmann space} (cf. \cite{hall}). 
Applying Proposition \ref{lem-repkernelrepresentation} to Theorem \ref{PWTranslation} we recapture (by a new method) a result of Haslinger (\cite{haslinger}) which represents the Bergman kernel of $\mathcal{U}_p$ in terms of reproducing kernels for $\mathcal{S}_p(t)$ as below. 
\begin{theorem}\label{thm-Haslinger}
For $t>0$, let the reproducing kernel for $\mathcal{S}_p(t)$ be denoted by $H_p(t;\cdot,\cdot)$. Then for $(z,w), (Z,W) \in \mathcal{U}_p$, the Bergman kernel $K$ of $\mathcal{U}_p$ is given by 
\begin{equation}\label{eq-Haslinger1}
K(z,w;Z,W) = 4\pi \int_{0}^{\infty} t H_p(t;w,W) e^{i2\pi (z-\overline{Z})t}\d t 
\end{equation}
\end{theorem}

To prove Theorem \ref{thm-Haslinger}, we will need the following lemma.

\begin{lemma}\label{lem-Hpt}
Let $t>0$ and let the reproducing kernel for $\mathcal{S}_p(t)$ be denoted by $H_p(t;\cdot,\cdot)$. Let $\widehat{\rho}_t : \C^n \to \C^n$ be as in \eqref{eq-rhohat}. Then for all $w, W \in \C^n$, we have
\begin{equation}\label{eq-Hpt}
H_p(t;w,W) = t^{1/\mu}H_p\left(1; \widehat{\rho}_t(w),\widehat{\rho}_t(W)\right), 
\end{equation}
where for the multi-index $m=(m_1,\dots,m_n)$ we let $1/\mu=\sum_{j=1}^n 1/m_j$.  
\end{lemma}

\begin{proof}
We first show that the map $D_t: \mathcal{S}_p(1)\to \mathcal{S}_p(t)$ given by 
\begin{equation}\label{eq-Spisom}
D_tf(\zeta) = t^{1/2\mu}f\left(\widehat{\rho}_t(\zeta)\right) 
\end{equation}
is an isometric isomorphism. By a standard change of variables argument applied to the linear change of coordinates $\zeta \mapsto \widehat{\rho}_t(\zeta)$ for $\C^n$ we see that
\begin{align}
\norm{f}_{\mathcal{S}_p(1)} &=  \norm{D_tf}_{\mathcal{S}_p(t)}. \label{eq-Tfnorm}
\end{align}
It follows from equation (\ref{eq-Tfnorm}) that 
$D_t$  is an isometry and hence injective.
For $F \in \mathcal{S}_p(t)$ let the function $f:\C^n\to \C$ be given by $f(\zeta) = t^{-1/2\mu} F\left(\widehat{\rho}_{1/t}(\zeta)\right)$. Then 
\[
D_tf(\zeta) = t^{1/2\mu}\cdot f(\widehat{\rho}_t(\zeta)) = t^{1/2\mu}\cdot t^{-1/2\mu}F\left(\widehat{\rho}_{1/t}(\widehat{\rho}_t(\zeta))\right) = F(\zeta).
\]
which shows that $D_t$ is surjective, and hence an isometric isomorphism.

Applying equation (\ref{eq-Rieszmap}) which defines the Riesz map to the functional $e_{\zeta}\circ D_t: \mathcal{S}_p(1) \to \C$ we obtain for all $f\in \mathcal{S}_p(1)$:
\begin{align*}
\left\langle f, R_{\mathcal{S}_p(1)}(e_{\zeta}\circ D_t)\right\rangle_{\mathcal{S}_p(1)} 
					= D_tf(\zeta) = f(\widehat{\rho}_t(\zeta))t^{1/2\mu} = \left\langle f, t^{1/2\mu} R_{\mathcal{S}_p(1)}(e_{\widehat{\rho}_t(\zeta)})\right\rangle_{\mathcal{S}_p(1)}.
\end{align*}
This shows  that
\[
R_{\mathcal{S}_p(1)}(e_{\zeta}\circ D_t) = t^{1/2\mu} R_{\mathcal{S}_p(1)}(e_{\widehat{\rho}_t(\zeta)}).
\]
Applying Proposition \ref{lem-repkernelrepresentation} to the isometric isomorphism $D_t$ in equation \eqref{eq-Spisom}, we see that the reproducing kernels $H_p(1;\cdot, \cdot)$ and $H_p(t;\cdot,\cdot)$ are related as
\begin{align*}
H_p(t;w,W) = \left \langle R_{\mathcal{S}_p(1)}(e_W\circ D_t), R_{\mathcal{S}_p(1)}(e_w\circ D_t)\right\rangle_{\mathcal{S}_p(1)} &= \left \langle t^{1/2\mu} R_{\mathcal{S}_p(1)}(e_{\widehat{\rho}_t(W)}), t^{1/2\mu} R_{\mathcal{S}_p(1)}(e_{\widehat{\rho}_t(w)}) \right\rangle_{\mathcal{S}_p(1)} \\
&= t^{1/\mu} H_p\left(1;\widehat{\rho}_t(w),\widehat{\rho}_t(W)\right), 
\end{align*}
where the last equality follows from \eqref{eq-repkernel}.
\end{proof}

\subsection{Proof of Theorem~\ref{thm-Haslinger}}

Let $T_S: \mathcal{H}_p \to A^2(\mathcal{U}_p)$ be the isometric isomorphism given by 
\[
T_Sf(z,w)= \int_0^{\infty} f(t,w)e^{i2\pi zt}\d t, \quad \textrm{for } (z,w) \in \mathcal{U}_p
\]as in Theorem \ref{thm-Haslinger}. We begin by showing that the image $R_{\mathcal{H}_p}\left(e_{(z,w)}\circ T_S\right)$ of the functional $e_{(z,w)}\circ T_S\in \mathcal{H}_p^*$ under the Riesz map $R_{\mathcal{H}_p}:\mathcal{H}_p^*\to \mathcal{H}_p$ is given by
\[
R_{\mathcal{H}_p}\left(e_{(z,w)}\circ T_S\right)(t,\zeta) = 4\pi t\overline{H_p(t;w,\zeta)}e^{-i2\pi \overline{z}t}.
\]
For each $(z,w) \in \mathcal{U}_p$, consider the function $\chi_{z,w}:(0,\infty)\times \C^n\to \C$ given by $\chi_{z,w}(t,\zeta)=4\pi t H_p(t;w,\zeta)e^{i2\pi zt}.$ Letting $z=x+iy$, we obtain
\begin{align*}
\norm{\chi_{z,w}}^2_{\mathcal{H}_p} &=\int_0^{\infty} \int_{\C^n} \abs{\chi_{z,w}(t,\zeta)}^2 \frac{e^{-4\pi p(\zeta)t}}{4\pi t} \d V(\zeta) \d t\\
         &= 4\pi \int_0^{\infty} te^{-4yt} \int_{\C^n} \abs{H_p(1;\widehat{\rho}_t(w),\widehat{\rho}_t(\zeta)}^2 t^{2/\mu}e^{-4\pi p(\widehat{\rho}_t(\zeta))} \d V(\zeta) \d t \; (\textrm{From lemma } \ref{lem-Hpt}) \\
        &= 4\pi \int_0^{\infty} t^{1+1/\mu}e^{-4\pi yt}\int_{\C^n} \abs{H_p(1;\widehat{\rho}_t(w),\widehat{\rho}_t(\zeta)}^2 e^{-4\pi p(\widehat{\rho}_t(\zeta))} \d V(\widehat{\rho}_t(\zeta)) \d t \\
        &= 4\pi \norm{H_p(1;\widehat{\rho}_t(w),\cdot)}_{\mathcal{S}_p(1)}^2 \int_0^{\infty} t^{1+1/\mu} e^{-4\pi yt}  \d t < \infty, 
\end{align*}
because $H_p(1;\widehat{\rho}_t(w),\cdot) \in \mathcal{S}_p(1)$ and the integral $\int_0^{\infty} t^{1+1/\mu} e^{-4\pi yt}\d t$ converges (as $y>0$).
This shows that $\chi_{z,w} \in \mathcal{H}_p$ and we have
\begin{align}
\langle f, \overline{\chi_{z,w}}\rangle_{\mathcal{H}_p} &= \int_0^{\infty} \int_{\C^n} f(t,\zeta) \left(4\pi tH_p(t,w,\zeta)e^{2\pi i zt}\right)\frac{e^{-4\pi p(\zeta)t}}{4\pi t}\d V(\zeta) \d t = \int_0^{\infty} f(t,w) e^{i2\pi zt}\d t \label{int2}, 
\end{align}
where equation (\ref{int2}) follows by the reproducing property of the kernel $H_p(t;w,\zeta)$. This shows that the map $R_{\mathcal{H}_p}\left(e_{(z,w)}\circ T_S\right)$ is given by 
\begin{equation}\label{eq-inner1}
R_{\mathcal{H}_p}\left(e_{(z,w)}\circ T_S\right)(t,\zeta) = \overline{\chi_{z,w}(t,\zeta)}= 4\pi t\overline{H_p(t;w,\zeta)}e^{-2\pi i \overline{z}t}.
\end{equation}

We now apply Proposition \ref{lem-repkernelrepresentation} to the isometric isomorphism $T_S$ above to obtain equation (\ref{eq-Haslinger1}).
Making use of identity \eqref{eq-inner1} and \eqref{eq-repkernel} in evaluating the inner product below gives
\begin{align*}
K(z,w;Z,W) &= \left\langle R_{\mathcal{H}_p}\left(e_{(Z,W)}\circ T_S\right),R_{\mathcal{H}_p}\left(e_{(z,w)}\circ T_S\right) \right\rangle_{\mathcal{H}_p}
=4\pi \int_0^{\infty} tH_p(t;w,W)e^{i2\pi (z-\overline{Z})t} \d t.
\end{align*}

\subsection{Another formula for the Bergman kernel of $\mathcal{U}_p$}

For $t\in \R$ let $\mathcal{Q}_p(t)$ be the weighted Bergman space on $\mathbb{B}_p$ with weight $w\mapsto \lambda(p(w),t)$, where $\lambda$ is as in equation \eqref{eq-lambdadef} i.e., 
\begin{equation}\label{eq-qptindex}
\mathcal{Q}_p(t)= A^2\left( \mathbb{B}_p,\lambda(p,t)\right).
\end{equation}
Then we may represent the Bergman kernel for $A^2(\mathcal{U}_p)$ in terms of reproducing kernel for $\mathcal{Q}_p(t)$ as below.

\begin{theorem}\label{thm-bergmankerneldilation}
Let $t\in \R$ and let the reproducing kernel for $\mathcal{Q}_p(t)$ be denoted by $X_p(t;\cdot,\cdot)$. Then for $(z,w), (Z,W) \in \mathcal{U}_p$, the Bergman kernel $K$ of $\mathcal{U}_p$ is given by 
\begin{equation}\label{eq-bergmankernel2}
K(z,w;Z,W) =  \int_{\R} X_p(t;w,W) \frac{z^{2\pi it}\cdot \overline{Z}^{-2\pi it}}{(z\overline{Z})^{1+1/2\mu}}\d t. 
\end{equation}
\end{theorem}
Recall that throughout this paper, $p$ is a nonnegative weighted homogeneous balanced polynomial with respect to the tuple $m= (m_1,\dots,m_n)$ of positive integers. With this notation let $M = \mathrm{l.c.m} (2,m_1,\cdots,m_n).$
Then we call a polynomial of the form 
\begin{equation}\label{eq-wthomopol}
g(w) = \sum_{\mathrm{wt}_m\alpha=\frac{k}{M}} C_{\alpha}w^{\alpha}, \; w\in \C^n
\end{equation}
a \emph{weighted homogeneous polynomial of weighted degree} $k/M$, where the sum is taken over multi-indices $\alpha \in \N^n$ whose weights with respect to the tuple $m$ are $k/M$. The following lemma will be needed to prove Theorem \ref{thm-bergmankerneldilation}.
\begin{lemma}\label{lem-homoseries}
A function $f\in \mathcal{O}(\mathbb{B}_p)$ admits a series expansion in weighted homogeneous polynomials, that is we may write
\begin{equation}\label{eq-homoseries}
f(z) = \sum_{k=0}^{\infty} f_k(z), \quad \text{for all } z \in \mathbb{B}_p
\end{equation}
where for each $k \geq 0$, $f_k$ is a weighted homogeneous polynomial of weighted degree $k/M$ as in \eqref{eq-wthomopol}.
The series \eqref{eq-homoseries} converges uniformly on compact subsets of $\mathbb{B}_p$.
\end{lemma}
\begin{proof}
For $\zeta=(\zeta_1,\dots,\zeta_n) \in \C^n$, let $\delta_{\zeta}:\C\to \C^n$  be the map
\begin{equation}\label{eq-deltadef}
\delta_{\zeta}(\mu) = \left(\mu^{M/2m_1}\zeta_1,\cdots, \mu^{M/2m_n}\zeta_n\right), \; \text{for } \mu \in \C.
\end{equation}
Note that since $M = \mathrm{l.c.m}\, (2,m_1,\dots,m_n)$, the powers of $\mu$ above are all positive integers, and a calculation shows that 
\begin{equation}\label{eq-pdelta}
p(\delta_{\zeta}(\mu)) = \abs{\mu}^{M}p(\zeta).
\end{equation}
Fix a $z\in \mathbb{B}_p\setminus\{0\}$, and let 
\[
\omega= \delta_z(1/\abs{z})= \left( \frac{z_1}{\abs{z}^{M/2m_1}},\cdots,\frac{z_n}{\abs{z}^{M/2m_n}}\right).
\]
We now show that the set of all $\mu \in \C$ such that $\delta_{\omega}(\mu) \in \mathbb{B}_p$ is given by  
$D(0,R(\omega))=\{ \mu\in \C\;|\; \abs{\mu} < R(\omega)\}$. By \eqref{eq-pdelta} we get
\[
p(\delta_{\omega}(\mu))= \abs{\mu}^Mp(\omega)=\abs{\mu}^Mp(\delta_z(1/\abs{z}))=\left(\frac{\mu}{\abs{z}}\right)^Mp(z),
\]
and if $\mu \in \C$ is  such that $\zeta = \delta_{\omega}(\mu)$ lies in $\mathbb{B}_p$, then
$p(\delta_{\omega}(\mu)) <1,$ 
which gives 
\begin{equation}\label{eq-Romegadef}
\abs{\mu} < \frac{\abs{z}}{(p(z))^{1/M}} = \frac{1}{(p(\omega))^{1/M}}:= R(\omega).
\end{equation}

Since $f$ is in $\mathcal{O}(\mathbb{B}_p)$, it admits a power series expansion which converges normally in an open set containing $0$, and rearranging the terms gives
\begin{equation}\label{eq-fseries}
f(\zeta) = \sum_{k=0}^{\infty} f_k(\zeta),
\end{equation}
where $f_k$ is a weighted homogeneous polynomial of weighted degree $k/M$. 
To show that the series in \eqref{eq-fseries} converges at the point $z\in \mathbb{B}_p$ we restrict $f$ to the complex analytic disc $\{\zeta \in \C^n\;|\; \zeta=\delta_{\omega}(\mu),\mu \in D(0,R(\omega))\}$ to obtain
\begin{equation}\label{eq-museries}
\varphi(\mu) := f(\delta_{\omega}(\mu)) =\sum_{k=0}^{\infty} f_k(\delta_{\omega}(\mu)) = \sum_{k=0}^{\infty} f_k(\omega) (\mu^M)^{k/M} = \sum_{k=0}^{\infty} f_k(\omega) \mu^k, 
\end{equation}
where we used the fact that $f_k$ is a weighted homogeneous polynomial of weighted degree $k/M$ to obtain the penultimate equality above. 
Since $f\in \mathcal{O}(\mathbb{B}_p)$, $\varphi$ is holomorphic in the disc $D(0,R(\omega))$. Now, it follows from equation \eqref{eq-Romegadef} that $\abs{z} < R(\omega)$, and consequently the series \eqref{eq-museries} converges for $\mu= \abs{z}$. Since $\varphi(\abs{z})=f(\delta_{\omega}(\abs{z}))=f(z)$, it follows that the series \eqref{eq-fseries} converges at the point $z$.

Now we wish to show that the series \eqref{eq-fseries} converges uniformly on compact subsets of $\mathbb{B}_p$. Suppose a compact subset $K$ of $\mathbb{B}_p$ is given. Then, there are numbers $0<s,q <1$ such that $K \subset \{ z\in \C^n \;|\; p(z) < q^Ms\}$. Letting $\omega=\delta_{z}(1/\abs{z})$ as before, we use \eqref{eq-pdelta} to get 
$p(z) = p(\delta_{\omega}(\abs{z})) = \abs{z}^M p(\omega).$
Then for every $z\in K$, we have $p(z) =\abs{z}^M p(\omega)<q^Ms$, from which it follows that 
\begin{equation}\label{eq-rdef}
\abs{z} < q\left(\frac{s}{p(\omega)}\right)^{1/M} := qr(\omega).
\end{equation}
Then we have
\begin{equation}\label{eq-fkz}
\abs{f_k(z)} = \abs{f_k(\delta_{\omega}(\abs{z}))}=\abs{f_k(\omega)}\left(\abs{z}^M\right)^{k/M} \leq \abs{f_k(\omega)}r^k(\omega)q^k.
\end{equation}
It follows from equations \eqref{eq-Romegadef} and \eqref{eq-rdef} that $r(\omega)<R(\omega)$. Thus, to estimate $\abs{f_k(\omega)}$, we note that it is the coefficient of $\mu^k$ in series \eqref{eq-museries} and apply Cauchy estimates to get  
$\abs{f_k(\omega)} \leq \frac{C}{r^k(\omega)},$ where $C= \max_{z\in K}\{\abs{f(z)}\}$.
Then equation \eqref{eq-fkz} reduces to 
$\abs{f_k(z)} \leq C q^k$. 
This shows that the series \eqref{eq-fseries} converges uniformly on $K$, since $0<q<1$ is independent of the choice of $z$. 
\end{proof}

\subsection{Proof of Theorem ~\ref{thm-bergmankerneldilation}} Let $T_V: \mathcal{X}_p \to A^2(\mathcal{U}_p)$ be the isometric isomorphism of Theorem \ref{PWDilation}. For each $(z,w)$ in $\mathcal{U}_p$, we first show that the image of the functional $e_{(z,w)}\circ T_V\in \mathcal{X}_p^*$ under the Riesz map $R_{\mathcal{X}_p}$ is given by
\[
R_{\mathcal{X}_p}\left(e_{(z,w)}\circ T_V\right)(t,\zeta) = \overline{X_p(t;w,\zeta)}\frac{\overline{z}^{-2\pi it}}{\overline{z}^{1+1/2\mu}}.
\]
Represent the map $e_{(z,w)}\circ T_V: \mathcal{X}_p \to \C$ in two different ways. One one hand, 
\begin{align}
\left(e_{(z,w)}\circ T_V\right)f &= T_Vf(z,w) \nonumber \\
					&= \int_{\R}  f\left(t, \frac{w_1}{z^{1/2m_1}},\cdots,\frac{w_n}{z^{1/2m_n}}\right)\frac{z^{2\pi it}}{z^{1+1/2\mu}}\d t \nonumber \\
                    &= \int_{\R} \left(\int_{\mathbb{B}_p} f\left(t, \frac{\zeta_1}{z^{1/2m_1}}, \cdots, \frac{\zeta_n}{z^{1/2m_n}}\right)X_p(t;w,\zeta) \lambda(p(\zeta),t)\d V(\zeta) \right)\frac{z^{i2\pi t}}{z^{1+1/2\mu}} \d t \label{eq-RXp1},
\end{align}
where the last equality follows from the reproducing property of the kernel $X_p(t;\cdot,\cdot)$.
On the other hand $R_{\mathcal{X}_p}\left(e_{(z,w)}\circ T_V\right)$, the image of the functional $e_{(z,w)}\circ T_V \in \mathcal{X}_p^*$ under the Riesz map $R_{\mathcal{X}_p}:\mathcal{X}_p^*\to \mathcal{X}_p$ is given by
\begin{align}
\left( e_{(z,w)}\circ T_V\right)f &= \left\langle f, R_{\mathcal{X}_p}\left(e_{(z,w)}\circ T_V\right)\right\rangle_{\mathcal{X}_p} \nonumber \\
								&= \int_{\R}\int_{\mathbb{B}_p} f\left(t,\frac{\zeta_1}{z^{1/2m_1}},\dots,\frac{\zeta_n}{z^{1/2m_n}}\right) \overline{R_{\mathcal{X}_p}\left(e_{(z,w)}\circ T_V\right)(t,\zeta)}\lambda(p(\zeta),t) \d V(\zeta) \d t. \label{eq-RXp2}
     \end{align}

Comparing \eqref{eq-RXp1} and \eqref{eq-RXp2} we claim that
\begin{equation}\label{eq-inner2}
R_{\mathcal{X}_p}\left(e_{(z,w)}\circ T_V\right)(t,\zeta)=\overline{X_p(t;w,\zeta)}\frac{\overline{z}^{-2\pi it}}{\overline{z}^{1+1/2\mu}}.
\end{equation} 
Let $\phi: \R\times\mathbb{B}_p \to \C$ be the difference of the two sides of the above equation, i.e.,  
\[
\phi(t,\zeta) = R_{\mathcal{X}_p}\left(e_{(z,w)}\circ T_V\right)(t,\zeta)-\overline{X_p(t;w,\zeta)}\frac{\overline{z}^{-2\pi it}}{\overline{z}^{1+1/2\mu}}.
\]
It then follows from equations (\ref{eq-RXp1}) and (\ref{eq-RXp2}) that the iterated integral 
\begin{equation}\label{eq-RXp3}
\int_{\R}\int_{\mathbb{B}_p}g\left(t,\frac{\zeta_1}{z^{1/2m_1}},\cdots,\frac{\zeta_n}{z^{1/2m_n}}\right) \overline{\phi(t,\zeta)}\lambda(p(\zeta),t) \d V(\zeta) \d t=0,
\end{equation}
for all $g\in \mathcal{X}_p$. We claim that this implies that $\phi \equiv 0$ which proves \eqref{eq-inner2}. 

The claim would be immediate if we knew that $\phi \in \mathcal{X}_p$ but since we don't, we proceed as follows. Let $g(t,\zeta) = h(t)q(\zeta)$, where $h \in \mathscr{C}_c(\R)$ is a compactly supported continuous function and $q$ is a holomorphic polynomial on $\C^n$. 
It is clear that $g\in \mathcal{X}_p$ and it follows from equation (\ref{eq-RXp3}) that 
\begin{align}
 \int_{\R} h(t)\int_{\mathbb{B}_p} q(\zeta) \overline{\phi(t,\zeta)}\lambda(p(\zeta),t) \d V(\zeta) \d t &=0. \label{eq-RXp4}
\end{align}
Since equation (\ref{eq-RXp4}) holds for all continuous compactly supported $h$, we must have for almost every $t\in \R$, 
\[
\int_{\mathbb{B}_p} q(\zeta) \phi(t,\zeta) \lambda(p(\zeta),t)\d V(\zeta)=0,
\]
Since $q$ is a polynomial and since polynomials are dense in $\mathcal{Q}_p(t)$ for all $t\in \R$, (by Lemma \ref{lem-homoseries}) it follows that for almost all $t\in \R$, 
$\phi(t,\zeta) \equiv 0 \; \textrm{on} \; \mathbb{B}_p$ and this proves the claim.

Applying Proposition (\ref{lem-repkernelrepresentation}) to the isometric isomorphism $T_V$ above, we get
\begin{align*}
K(z,w;Z,W) &= \left\langle R_{\mathcal{X}_p}\left(e_{(Z,W)}\circ T_V\right),R_{\mathcal{X}_p}\left(e_{(z,w)}\circ T_V\right) \right\rangle_{\mathcal{X}_p}
=\int_{\R} X_p(t;w,W) \frac{z^{2\pi it}\overline{Z}^{-2\pi it}}{(z\overline{Z})^{1+1/2\mu}}\d t,
\end{align*}
where we used \eqref{eq-inner2} and \eqref{eq-repkernel} to compute the inner product above.

\providecommand{\bysame}{\leavevmode\hbox to3em{\hrulefill}\thinspace}
\providecommand{\MR}{\relax\ifhmode\unskip\space\fi MR }
\providecommand{\MRhref}[2]{%
  \href{http://www.ams.org/mathscinet-getitem?mr=#1}{#2}
}
\providecommand{\href}[2]{#2}

\end{document}